\documentclass[12pt]{amsart}

\usepackage[dvips]{graphics}
\usepackage{epsfig}
\usepackage{amssymb}
\usepackage{amsthm}
\usepackage{amscd}
\usepackage[mathscr]{eucal}
\usepackage{enumitem}

\usepackage{hyperref}

\newtheorem{Theorem}{Theorem}[section]
\newtheorem{theorem}[Theorem]{Theorem}

\newtheorem{proposition}[Theorem]{Proposition}

\newtheorem{corollary}[Theorem]{Corollary}

\newtheorem{lemma}[Theorem]{Lemma}
\newtheorem{Fact}[Theorem]{Fact}
\newtheorem{fact}[Theorem]{Fact}
\newtheorem{remark/def}[Theorem]{Remark/Definition}

\newtheorem{claim}[Theorem]{Claim}

\theoremstyle{definition}

\newtheorem{remark}[Theorem]{Remark}

\newtheorem{definition}[Theorem]{Definition}

\newtheorem{notation}[Theorem]{Notation}

\newtheorem{note}[Theorem]{Note}
\newtheorem{Question}[Theorem]{Question}
\newtheorem{question}[Theorem]{Question}
\newtheorem{def/rem}[Theorem]{Definition/Remark}
\newtheorem{not/rem}[Theorem]{Notation/Remark}
\newsavebox{\indbin}
\savebox{\indbin}{\begin{picture}(0,0)
\newlength{\gnu}
\settowidth{\gnu}{$\smile$} \setlength{\unitlength}{.5\gnu}
\put(-1,-.65){$\smile$} \put(-.25,.1){$|$}
\end{picture}}
\def \indo {\mathop{\smile \hskip -0.9em ^| \ }}

\newcommand{\be}{\begin{enumerate}}
\newcommand{\bd}{\begin{defn}}

\newcommand{\bt}{\begin{theorem}}
\newcommand{\bl}{\begin{lemma}}
\newcommand{\ee}{\end{enumerate}}
\newcommand{\ed}{\end{defn}}
\newcommand{\et}{\end{theorem}}
\newcommand{\el}{\end{lemma}}

\newcommand{\sm}{\setminus}

\newcommand{\CP}{{\mathcal P}}

\newcommand{\CA}{{\mathcal A}}
\newcommand{\CB}{{\mathcal B}}
\newcommand{\CC}{{\mathcal C}}
\newcommand{\CE}{{\mathcal E}}

\newcommand{\CL}{{\mathcal L}}

\newcommand{\CM}{{\mathcal M}}

\newcommand{\CU}{{\mathcal U}}
\newcommand{\BN}{{\mathbb N}}
\newcommand{\BZ}{{\mathbb Z}}
\newcommand{\BQ}{{\mathbb Q}}
\newcommand{\BR}{{\mathbb R}}

\newcommand{\dom}{\mbox{dom}}

\newcommand{\id}{\operatorname{id}}
\newcommand{\Aut}{\operatorname{Aut}}
\newcommand{\aut}{\operatorname{Aut}}
\newcommand{\Autf}{\operatorname{Autf}}
\newcommand{\autf}{\operatorname{Autf}}

\newcommand{\LL}{\operatorname{L}}
\newcommand{\gal}{\operatorname{Gal}}
\newcommand{\gall}{\gal_{\LL}}

\newcommand*{\upSmallFrown}{\mathbin{\raisebox{0.9ex}{$\smallfrown$}}}

\newcommand{\sdist}{\operatorname{Sd}}

\newcommand{\Th}{\operatorname{Th}}

\def\fix{\operatorname{fix}}
\def\res{\operatorname{res}}

\def\eq{\operatorname{eq}}

\def\cl{\operatorname{cl}}

\def\dcl{\operatorname{dcl}}

\def\dom{\operatorname{dom}}

\def\acl{\operatorname{acl}}

\def\ker{\operatorname{Ker}}

\def\Im{\operatorname{Im}}

\def\tp{\operatorname{tp}}

\def\stp{\operatorname{stp}}

\def\lstp{\operatorname{Lstp}}
\def\Lstp{\operatorname{Lstp}}

\def\x{\bar{x}}
\def\y{\bar{y}}
\def\z{\bar{z}}
\def\a{\bar{a}}
\def\b{\bar{b}}

\def\p{\bar{p}}

\def\raw{\rightarrow}

\def\sm{\setminus}

\def\mor{\operatorname{Mor}}
\def\supp{\operatorname{supp}}
\def\Bd{\partial}

\title{The Lascar groups and the first homology groups  in model theory}

\author{Jan Dobrowolski, Byunghan Kim, and Junguk Lee}

\address{Department of Mathematics\\ Yonsei University\\
50 Yonsei-Ro, Seodaemun-Gu\\
Seoul 120-749, Korea}

\email{dzas87@gmail.com}
\email{bkim@yonsei.ac.kr}
\email{ljw@yonsei.ac.kr}

\thanks{All authors were supported by Samsung Science Technology Foundation under Project Number
SSTF-BA1301-03.
The third author was also supported by NRF of Korea grant 2013R1A1A2073702 and the Yonsei University Research Fund(Post Doc. Researcher Supporting Program) of 2016(project no.:2016-12-0004)}
\begin{document}

\begin{abstract}  Let $p$ be a strong type of an algebraically closed tuple over $B=\acl^{\eq}(B)$ in any theory $T$.
Depending on a ternary relation $\indo^*$ satisfying some basic axioms (there is at least one such, namely the trivial independence in $T$),    the first homology
group  $H^*_1(p)$   can be introduced, similarly to  \cite{GKK1}.

We show that there is  a canonical surjective homomorphism from the Lascar group over $B$ to $H^*_1(p)$. We also notice that the map factors naturally via a surjection from the `relativised' Lascar group of the type (which we define in analogy with
the Lascar group of the theory) onto the homology group, and we give an explicit description of its kernel.
Due to this characterization, it follows that the first homology group of $p$ is independent from the choice of $\indo^*$, and can be written simply as
$H_1(p)$.

As consequences, in any $T$,  we show that
$|H_1(p)|\geq 2^{\aleph_0}$ unless $H_1(p)$ is trivial, and
 we give a criterion for the equality of stp and Lstp of algebraically closed tuples
using the notions of the first homology group and a relativised Lascar group.

We also argue how any abelian connected compact group can appear as the first homology group of the type of a model. 
\end{abstract}

\maketitle

In this paper we study the first homology group of a strong type in any theory. 

Originally,
in \cite{GKK1} and \cite{GKK}, a homology theory only for rosy theories is developed. Namely, given a strong type $p$ in a rosy theory $T$,
the notion of the $n$th homology group $H_n(p)$ depending on thorn-forking independence relation  is introduced.  Although the homology groups are  defined analogously
as in singular homology theory in algebraic topology, the $(n+1)$th homology group   for $n>0$ in the rosy theory context  has to do with the $n$th homology group in algebraic topology. For example as in \cite{GKK1},
$H_2(p)$ in stable theories has to do with  the  fundamental group in topology. This implies that, already in rosy theories, $H_1(p)$ is detecting somewhat endemic properties of $p$ existing only in model theory context.

Indeed, in every known rosy example, $H_n(p)$  for $n\geq 2$ is a profinite abelian group. In \cite{GKK2}, it is proved to be so  when $T$ is stable under a canonical condition, and conversely, every profinite abelian group can arise in this form.
On the other hand,  we show in this paper that the first homology groups appear to have distinct features as follows.

Let $p=\tp(a/B)$ be a  strong type over $B=\acl^{\eq}(B)$ in any theory $T$.
Fix a ternary invariant independence relation $\indo^*$ among small sets satisfying finite character, normality, symmetry, transitivity and extension. (There is at least one such relation, by putting $A\indo_CD$ for any sets $A,C,D$.) Then we can analogously define $H^*_1(p)$ depending on $\indo^*$, (which of course is the same as $H_1(p)$ when
$\indo^*$ is thorn-independence in rosy $T$). In this note,   a  canonical epimorphism from the Lascar group over $B$ of $T$ to $H^*_1(p)$ is constructed.
Indeed, we also introduce the notion of the relativised Lascar group of a type which  is proved to be independent from the  choice of the monster model of $T$, and
the homomorphism factors through a surjection from the relativised Lascar group of $\p=\tp(\acl(aB)/B)$ onto $H^*_1(p)$. Moreover,
we can identify its kernel. Roughly,
$H^*_1(p)$ has to do with the abelianization of the relativised Lascar group of $\p$.
More precisely, $H^*_1(p)=G/K$, where $G$ is the  group of automorphisms of the realization set of $\p$, and $K$ is the normal subgroup of $G$ fixing each orbit
under the action of the derived subgroup of $G$. Surprisingly, this conclusion is independent from the choice of $\indo^*$ satisfying the axioms.\footnote{However it is not clear whether the same feature can happen for $n$th homology groups of types for $n>1$.}
 Hence, we can write the first homology group simply as  $H_1(p)$, which makes sense in any theory.

Consequently, we show that $|H_1(p)|\geq 2^{\aleph_0}$ unless $H_1(p)$ is trivial,
and exhibit a non-profinite  example in a rosy theory. In conclusion, we find a
criterion for the coincidence of notions of strong types and Lascar types of algebraically closed tuples
in any theory,  in terms of the triviality of the first homology groups and the abelianness of a relativised Lascar group (Corollary \ref{equality_stp_Lstp}).
It seems reasonable to ask whether this criterion can be applied in verifying or refuting stp$\equiv$Lstp in simple theories.

\medskip

In Section 1, we introduce/recall basic definitions of the first homology group of  a strong type for any theory. In Section 2, as mentioned above,
 we construct a surjective homomorphism from the  Lascar group to the first homology group. In Section 3,
we introduce the aforementioned concept of   relativized Lascar groups,  and
in Section 4, we prove the characterization theorem (Theorem \ref{kernel_derived}) of the first homology group and give a criterion for Lstp$\equiv$stp.
We also argue that
   the size of the first homology group of a strong type is either $1$ or $\geq 2^{\aleph_0}$ (in Theorem \ref{cardinality_h1}, and a more detailed explanation is given in Section 6).
   In Section 5, we state that any connected compact abelian group can appear as the first homology group of the type of a model (Theorem \ref{ab_conn_gp_h1}), which follows from a result by Bouscaren, Lascar, Pillay, and Ziegler. 
    We also give a more precise  example of a type in a  rosy theory having a non-profinite  first homology group. 
   We point out here that this paper is a result of merging two notes. The first one, single-authored by Junguk Lee, covered Section 2, Theorem \ref{cardinality_h1} in Section 4, and Section 5.2, and the second note, jointly written by the three authors,  consisted of Section 1,3, 4 and Section 5.1.


\section{Introduction}


Throughout this paper, we work in a large saturated model $\CM(=\CM^{\eq})$ of a complete theory
$T$, and we use standard notations.
For example, unless stated  otherwise, $a,b,\ldots,$ and  $A, B,\ldots$  are small but possibly infinite
tuples and sets from
$\CM$, respectively, and  $a\equiv_Ab$, $a\equiv^s_Ab$,  $a\equiv^L_Ab$ mean
$\tp(a/A)=\tp(b/A)$, $\stp(a/A)=\stp(b/A)$, $\lstp(a/A)=\lstp(b/A)$,
respectively.
  For the general theory of   model theory, of the Lascar groups, and of rosy theories, we refer to \cite{K},
  \cite{Z}, and \cite{EO}, respectively. For the homology theory in model theory, see \cite{GKK,GKK1}.
  A particular case of the first homology group with respect to thorn-forking in rosy theories is studied in
  \cite{KKL},\cite{KL}. The main difference of the first homology groups introduced in this section from those in the references is
  that the groups are defined with respect to a fixed independence notion in an arbitrary theory as follows, not necessarily thorn-forking/Shelah-forking in rosy/simple theories. 
  However, as the reader will see, all the arguments from the rosy theory
  context can follow in the general context.


\medskip

For the rest of this section (and also for Section \ref{characterization}), we fix a ternary
automorphism-invariant relation $\indo^*$ between
small sets of $\CM$ satisfying
\begin{itemize}
\item
finite  character: for any sets $A,B,C$, we have $A\indo^*_CB$ iff $a\indo^*_Cb$ for any finite tuples $a\in A$ and $b\in B$;

\item normality: for any sets $A$, $B$ and $C$, if $A\indo^* _C B$, then $A\indo^*_C \acl(BC)$;

\item
symmetry: for any sets $A,B,C$, we have $A\indo^*_CB$ iff $B\indo^*_CA$;
\item
transitivity: $A\indo^*_BD$ iff $A\indo^*_BC$ and $A\indo^*_CD$,  for any sets $A$ and
$B\subseteq	 C\subseteq D$;
\item
 extension: for any sets $A$ and $B\subseteq C$, there is $A'\equiv_BA$ such that
 $A'\indo^*_BC$.
\end{itemize}

Throughout this paper we call the above axioms {\bf the basic 5 axioms}.
We say that $A$ is $*$-independent from $B$ over $C$ if $A\indo^*_C B$. Notice that there is at least one such relation for any theory, namely,
the {\em trivial independence relation}: For any sets $A,B,C$, put $A\indo^*_BC$. Of course there is
a non-trivial such relation when $T$ is simple or rosy, given by forking or thorn-forking, respectively.

Now, we also fix a strong type $p$ of possibly infinite arity over $B=\acl(B)$. We shall define the first homology group of $p$ with respect to $\indo^*$, analogously to that in the references.
 Hence we begin by  recalling  some notations from the references.

\begin{notation}
Let $s$ be an arbitrary finite set of natural numbers. Given any subset
$X\subseteq \CP(s)$, we may view $X$ as a category where for any
$u, v \in X$, $\mor(u,v)$ consists of a single morphism 
$\iota_{u,v}$ if $u\subseteq v$, and $\mor(u,v)=\emptyset$ otherwise. 
If $f\colon X \raw \CC$ is any functor into some category $\CC$, then for any $u, v\in X$ with $u \subseteq v$, we let $f^u_v$ denote the morphism
$f(\iota_{u,v})\in \mor_{\CC}(f(u),f(v))$.
We shall call $X\subseteq \CP(s)$ a \emph{primitive category} if $X$ is non-empty and  \emph{downward closed}; i.e.,  for any $u, v\in \CP(s)$,  if $u\subseteq v$ and $v\in X$ then $u\in X$. (Note that all primitive categories have  the empty set $\emptyset\subset \omega$ as an object.)

We use now $\CC_B$ to denote the category  whose objects are all the small subsets of $\CM$ containing $B$, and whose morphisms are elementary maps over $B$.
For a functor $f:X\to \CC_B$ and objects $u\subseteq v$ of $X$, $f^u_v(u)$ denotes the set $f^u_v(f(u))(\subseteq f(v))$.
\end{notation}

\begin{definition}\label{p-*-functors}
By a \emph{$*$-independent functor in $p$}, we mean a functor $f$ from some primitive category $X$ into $\CC_B$ satisfying the following:

\be
\item If $\{ i \}\subset \omega$ is an object in $X$, then $f(\{ i \})$ is of the form $\acl(Cb)$ where  $b\models p$, $C=\acl(C)=f^{\emptyset}_{\{ i \}}(\emptyset)\supseteq B$,  and $b\indo^*_B C$.

\smallskip

\item Whenever $u(\neq \emptyset)\subset \omega$ is an object in $X$, we have
\[ f( u) = \acl \left( \bigcup_{i\in u} f^{\{ i \}}_u(\{ i \}) \right)
\]
and $\{f^{\{i\}}_u(\{i\})|\ i\in u \}$ is $*$-independent over $f^\emptyset_u(\emptyset)$.
\ee
 We let $\CA^*_p$  denote the family of all $*$-independent functors in $p$.

 A $*$-independent functor $f$  is called a {\em $*$-independent $n$-simplex} in $p$  if $f(\emptyset)=B$ and $\dom(f)=\CP(s)$ with $s\subset \omega$ and  $|s|=n+1$.  We call $s$ the {\em support} of $f$ and denote it by $\supp(f)$.
\end{definition}

In the rest we may call a  $*$-independent $n$-simplex in $p$ just  an {\em $n$-simplex} of $p$, as far as no  confusion arises.
We are ready to define the first homology group $H^*_1(p)$ of $p$ depending on our choice of the independence relation $\indo^*$.

\begin{definition} Let $n\geq 0$.
 We define:
\begin{align*}
&S_n(\CA^*_p):= \{\, f\in \CA^*_p \mid \mbox{$f$ is an $n$-simplex  of $p$}\, \}\\
&C_n(\CA^*_p): =  \mbox{the free abelian group generated by $S_n(\CA^*_p)$.}
\end{align*}
An element of $C_n(\CA^*_p)$ is called an \emph{$n$-chain} of $p$.
The support of a chain
$c$, denoted by $\supp(c)$, is the union of the supports of all the simplices that appear in $c$ with a non-zero
coefficient.
Now for $n\geq 1$ and  each $i=0, \ldots,  n$, we  define a group homomorphism
\[ \Bd^i_n\colon C_n(\CA^*_p) \raw C_{n-1}(\CA^*_p)
\]
by  putting, for any $n$-simplex   $f\colon \CP(s) \raw \CC$ in $S_n(\CA^*_p)$ where  $s = \{ s_0 <\cdots < s_n \}\subset \omega$,
\[ \Bd_n^i(f):= f\restriction  \CP(s\setminus\{s_i\})
\]
and then extending linearly to all $n$-chains in $C_n(\CA^*_p)$. Then we define  the \emph{boundary map}
\[ \Bd_n\colon C_n(\CA^*_p)\raw C_{n-1}(\CA^*_p)
\]
by
\[ \Bd_n(c):=\sum_{0\leq i\leq n}(-1)^i \Bd_n^i(c).
\]
We shall often refer to $\Bd_n(c)$ as the \emph{boundary of $c$}. Next, we define:
\begin{align*}
&Z_n(\CA^*_p):= \ker \, \Bd_n\\
& B_n(\CA^*_p):= \Im \, \Bd_{n+1}.
\end{align*}

\noindent The elements of $Z_n(\CA^*_p)$ and $B_n(\CA^*_p)$ are called \emph{$n$-cycles} and \emph{$n$-boundaries} in $p$, respectively.
It is straightforward to check that
$ \Bd_{n}\circ \Bd_{n+1} =0.$
Hence we can now define the group
\[ H^*_n(p):= Z_n(\CA^*_p)/B_n(\CA^*_p)
\]
called the \emph{$n$}th \emph{$*$-homology group} of $p$.
\end{definition}

\begin{notation}
\be
\item
For $c\in Z_n(\CA^*_p)$, $[c]$ denotes the homology class of $c$ in $H^*_n(p)$.
\item When $n$ is clear from the context, we shall often omit $n$ in $\Bd^i_n$ and in $\Bd_n$, writing simply as  $\Bd^i$ and $\Bd$.
\ee
\end{notation}

\begin{definition}
A $1$-chain  $c\in C_1(\CA^*_p)$ is called a \emph{$1$-$*$-shell} (or just a $1$-shell) in $p$ if it is of the form
\[c =f_0-f_1+f_2
\]
where $f_i$'s are $1$-simplices of $p$ satisfying
\[ \Bd^i f_j = \Bd^{j-1} f_i \quad \mbox{whenever $0\leq i < j \leq 2$.}
\]
Hence, for $\supp(c)=\{n_0<n_1<n_2\}$ and $ k\in \{0,1, 2\}$, it follows that
$$\supp(f_k)=\supp(c)\smallsetminus \{n_k\}.$$
\end{definition}

Notice that the boundary of any $2$-simplex is a $1$-shell.

\begin{not/rem}\label{p and bar p}
Let $p(x)=\tp(a/B)$ be fixed, and let $\p(\x)=\tp(\acl(aB)/B)$ (with some enumeration of $\acl(aB)$). Obviously $\p(\x)$ only depends on $p$ (not on its
realizations).
By the definitions of the $*$-independent functors and the first homology group, $H_1^*(p)$ and $H_1^*(\p)$ are identical.

If $c$ is a $1$-shell, then in $H^*_1(p)$,  we shall see in Remark \ref{summary_H1_representation} that
$[-c]=[c']$ where $c'$ is another $1$-shell with $\supp(c')=\supp(c)$.

We note now that in \cite{GKK}, the notion of  an {\em amenable} collection of functors into
a category is introduced, and due to the 5 axioms that $\indo^*$ satisfies, it is clear that $\CA^*_p$ forms such a collection of functors into $\CC_B$.
Therefore the following corresponding fact holds. 
\end{not/rem}

\begin{fact}\label{H1 in amenable family} (\cite{GKK1} or \cite{GKK})
\[ H^*_1(p) = \{ [c] \mid c \mbox{ is a } \mbox{$1$-$*$-shell with}\ \supp(c) = \{ 0, 1, 2 \} \, \}.
\]
\end{fact}

So if any $1$-shell is the boundary of  some $2$-chain then $H^*_1(p)=0$.\footnote{Notice that in Definition \ref{p-*-functors}, we take algebraic closures, rather that bounded closures. Hence even  when $\indo^*$ is nonforking in simple $T$, it is not known whether $H^*_1(p)$ is trivial. We put this as an open question
in Question \ref{simpleh_1}.}

\begin{remark}\label{p is bar p}
The following Fact \ref{h1=0} directly comes from   \cite[Theorem 2.4]{KKL} (and above \ref{p and bar p}), since
the proof of the theorem only uses the fact that thorn-independence in any rosy theory satisfies the basic 5
axioms. But we point out that  corrections should be made  in the theorem and other
results in \cite{KKL}. Namely, $p(x)$ there should be changed to $\bar p(\bar x)$ since a vertex of a simplex in $p$ is
an algebraically closure over $B$ of a realization of $p$. In fact it is not clear whether $p(x)$ 
being a Lascar type implies that $\p(\x)$ is also a Lascar type
(the converse always holds though),   unless $T$ is G-compact over $B$:
Let $T$ be G-compact over $B$, and let $p$ be a Lascar type; i.e., for any $a,b \models p$, 
we have $a \equiv^L_B b$. We claim that $\p$ is a Lascar type, too.
Since $T$ is G-compact over $B$, equality of $\Lstp$ over $B$ is $B$-type-definable and a conjunction of those of   finite arities. 
Thus by compactness  for $a\models p $ and finite $c\in \acl(aB)$, it suffices to show that $q(x,y):=\tp(ac/B)$ is a Lascar type. Now since $p$ is a Lascar type, for any
$a'\models p$ there is $c'\in \acl(a'B)$ such that $ac\equiv^L_B a'c'$. Since there are only finitely many conjugates of $c'$ over $a'B$, it follows that there are at most
finitely many distinct Lascar types in $q$. This implies that equality of $\Lstp$  in $q$ is  relatively definable over $B$. But since $q$ is a strong type
over $B=\acl(B)$, there is only one Lascar class in $q$.
\end{remark}


\begin{fact}\label{h1=0}
Suppose that $\p(\x)$ is a Lascar type. Then $H_1^*(p)=0$.
\end{fact}

For the rest of this paper, for notational simplicity,  we suppress $B$ to $\emptyset$ by naming it (and reuse $B$ to mean an arbitrary small set).
In particular, $\CC$ denotes $\CC_B$. {\bf We further suppose (until the end of Section 4)
that the fixed strong type $p$ is a type of an
algebraically closed set (by assuming $p=\p$) or that the algebraic closure of its realization is the same as its definable closure.} This process is necessary as pointed out in Remark \ref{p is bar p}, and will not affect in computing $H^*_1(p)$ due to \ref{p and bar p}.

\section{Lascar groups and the first homology groups}
In this section, we show that there is a canonical epimorphism from the Lascar group
$\gal_L(T)$ onto the first homology group $H_1^*(p)$ of $p$.

\subsection{Representations of 1-shells}

\begin{definition}\label{representation}
\be\item
We introduce some notation which will be used throughout. Let $f\colon \CP(s) \raw \CC$ be an $n$-simplex in $p$. For $u\subset s$ with $u = \{ i_0 <\ldots <i_k \}$, we shall write $f(u)=[a_0 \ldots a_k]_u$, where each $a_j\models p$ is 
an algebraically closed tuple as assumed before, if $f(u)=\acl(a_0\ldots a_k)$, and $\acl(a_j)=f^{ \{ i_j \} } _u (\{i_j\})$. So, $\{a_0,\ldots,a_k\}$ is $*$-independent.
Of course, if we write $f(u)\equiv [b_0 \ldots b_k]_u$, then it means that there is an automorphism sending $a_0\ldots a_k$ to $b_0\ldots b_k$. 
\item
Let $s=f_{12}-f_{02}+f_{01}$ be a $1$-$*$-shell in $p$ such that $\supp(f_{ij})=\{n_i,n_j\}$ with $n_i<n_j$ for $0\le i<j\le 2$. Clearly there is a quadruple $(a_0,a_1,a_2,a_3)$ of realizations of $p$ such that $f_{01}(\{n_0,n_1\})\equiv[a_0 a_1]_{\{n_0,n_1\}}$, $f_{12}(\{n_1,n_2\})\equiv[a_1 a_2]_{\{n_1,n_2\}}$, and $f_{02}(\{n_0,n_2\})\equiv[a_3 a_2]_{\{n_0,n_2\}}$. We call this quadruple a {\em  representation of $s$}. For  any such representation of $s$,  call $a_0$  an {\em  initial point}, $a_3$ a {\em terminal point},  and $(a_0,a_3)$ an {\em  endpoint pair} of the representation.
\ee
\end{definition}

In the next theorem, we will see that the endpoint pairs of representations determine the classes of $1$-shells in $H_1^*(p)$, and the group structure of $H_1^*(p)$ can be described by endpoint pairs.
\begin{theorem}\label{endpt_class}
Let $s_0$ and $s_1$ be $1$-shells with support $\{0,1,2\}$.
Suppose they have some representations with the same endpoint pair.
Then $s_0-s_1$ is a boundary of a $2$-chain, that is, $s_0$ and $s_1$ are in the same homology class in $H_1^*(p)$
\end{theorem}
\begin{proof}
Let $s_k:=f_{12}^k-f_{02}^k+f_{01}^k$ for $k=0,1$, where for each $1$-shell $f_{ij}^k$ in $p$, $\supp(f_{ij}^k)=\{i,j\}$. Suppose $s_0$ and
$s_1$ have representations $(a,b_0,c_0,a')$ and $(a,b_1,c_1,a')$, respectively. Take
$b\models p$ such that $b\indo^* ab_0b_1c_0c_1a'$. Then,
there is a $2$-chain $\alpha=(a^0_{01}+a^0_{12}-a^0_{02})-(a^1_{01}+a^1_{12}-a^1_{02})$,
where for each $k=0,1$ and $0\le i<j\le 2$, $a^k_{ij}$ is a  $2$-simplex satisfying the following:
For $k=0,1$,

\be
	\item $\supp(a_{ij}^k)=\{i,j,3\}$;
	\item $a^k_{ij}\restriction \CP(\{i,j\})=f^k_{ij}$; and
	\item $a^k_{01}(\{0,1,3\})=[a b_k b]_{\{0,1,3\}}$,
	$a^k_{12}(\{1,2,3\})=[b_k c_k b]_{\{1,2,3\}}$,\\ $a^k_{02}(\{0,2,3\})=[a'c_k b]_{\{0,2,3\}}$,
\ee

\noindent and $\Bd(a^k_{01}+a^k_{12}-a^k_{02})=s_k-(f_k-g_k)$, where
$f_k=a^k_{01}\restriction\CP(\{0,3\})$ and $g_k=a^k_{02}\restriction\CP(\{0,3\})$.
So $\Bd\alpha=(s_0-s_1)-((f_1-g_1)-(f_0-g_0))$. It is enough to show that a
$1$-chain $(f_1-g_1)-(f_0-g_0)$ is a boundary of a $2$-chain. (Notice that even if for $f_0, f_1$ (similarly for $g_0,g_1$),
$f_0(\{0,3\})=f_1(\{0,3\})=[ab]_{\{0,3\}}$, it need not be that $f_0=f_1$ since $f_0(\{0\})=f^0_{01}(\{0\})$ and $f_1(\{0\})=f^1_{01}(\{0\})$.)

Now by the extension axiom, we can choose $c,c'\models p$ such that
$c\indo^* ab$, $c'\indo^* a'b$, and $ca\equiv c'a'$. For $k=0,1$, consider
$2$-simplices $\hat f_k$ and $\hat g_k$  such that:
\be
	\item $\supp(\hat f_k)=\supp(\hat g_k)=\{0,3,4\}$;
	\item $\hat f_k(\{0,3,4\})=[abc]_{\{0,3,4\}}$ and $\hat g_k(\{0,3,4\})=[a'bc']_{\{0,3,4\}}$;
	\item $\Bd^0 \hat f_0=\Bd^0 \hat f_1$, $\Bd^0 \hat g_0=\Bd^0 \hat g_1$;
	\item $\Bd^1 \hat f_k=\Bd^1 \hat g_k$; and
	\item $\Bd^2 \hat f_k =f_k$, $\Bd^2 \hat g_k=g_k$.
\ee
Then the $1$-chain $(f_1-g_1)-(f_0-g_0)$ is the boundary of the $2$-chain $(\hat f_1 -\hat g_1)-(\hat f_0-\hat g_0)$.
\end{proof}

\begin{remark}\label{rmkforendpt_class}
Notice that in the  proof above, that $s_0,s_1$ have the same support is not at all essential, so Theorem  \ref{endpt_class}
still holds even if their supports are distinct. Moreover, if $(a,b,c,a)$ is a representation of some $1$-shell $s$, even if
$a$ and $bc$ need not be $*$-independent, $[s]=0$ in $H^*_1(p)$. 
\end{remark}

\begin{theorem}\label{endpt_gpstr}
Let $s_0$ and $s_1$ be $1$-shells with support $\{0,1,2\}$, and let $a,a',a''\models p$ be
such that $(a,a')$ and $(a',a'')$ are the endpoint pairs of representations of $s_0$ and
$s_1$, respectively.
Then there is a $1$-shell $s$ with support $\{0,1,2\}$ having a representation with the endpoint
pair
$(a,a'')$ such that $[s]=[s_0]+[s_1]$ in $H_1^*(p)$.
\end{theorem}
\begin{proof}
Let $s_0=f^0_{01}+f^0_{12}-f^0_{02}$, and let  $s_1=f^1_{01}+f^1_{12}-f^1_{02}$ with  $\supp(f^k_{ij})=\{i,j\}$.
Since $(a,a')$ and $(a',a'')$ are some endpoint pairs of $s_0$ and $s_1$ respectively, by Theorem \ref{endpt_class},
we may assume the restrictions of $f^0_{01}$ and $f^1_{01}$ to the domain $\CP(\{0\})$ are
the
same, hence so are the restrictions of $f^0_{02}$ and $f^1_{02}$ to $\CP(\{0\})$.
 Let
$b_0, b_1, c_0, c_1\models p$
be such that the two quadruples $(a,b_0,c_0,a')$ and $(a',b_1,c_1,a'')$ are representations
of $s_0$ and $s_1$, respectively. Consider two $*$-independent elements
$d,e\models p$ with $de\indo^* aa'a''b_0 b_1 c_0 c_1$. Then there is a $2$-chain
$\alpha=(a^0_{01}+a^0_{12}-a^0_{02})-b+(a^1_{01}+a^1_{12}-a^1_{02})$, where for $k=0,1$
and $0\le i<j\le 2$, $a^k_{ij}$ and $b$ are $2$-simplices satisfying the following:
\be
	\item $\supp(a^k_{ij})=\{i,j,3+k\}$ and $\supp(b)=\{0,3,4\}$;
	\item $a^k_{ij}\restriction \CP(\{i,j\})=f^k_{ij}$;
	\item $a^0_{01}(\{0,1,3\})=[ab_0 d]_{\{0,1,3\}}$, $a^0_{12}(\{1,2,3\})=[b_0 c_0 d]_{\{0,2,3\}}$, $a^0_{02}(\{0,2,3\})=[a' c_0 d]_{\{0,2,3\}}$, $a^1_{01}(\{0,1,4\})=[a' b_1 e]_{\{0,1,4\}}$, $a^1_{12}(\{1,2,4\})=[b_1 c_1 e]_{\{1,2,4\}}$, $a^1_{02}(\{0,2,4\})=[a'' c_1 e]_{\{0,2,4\}}$, and $b(\{0,3,4\})=[a' d e]_{\{0,3,4\}}$;
\ee
\noindent and $\Bd(\alpha)=s_0+s_1-s'$, where
$s'=a^0_{01}\restriction \CP(\{0,3\})+b\restriction \CP(\{3,4\})-a^1_{02}\restriction \CP(\{0,4\})$ is a
$1$-shell of a support $\{0,3,4\}$ having $(a,a'')$ as its endpoint pair. 
Now by Remark   \ref{rmkforendpt_class}, for a $1$-shell $s$ of a support $\{0,1,2\}$ obtained from $s'$  by simply changing the support $\{0,3,4\}$ to $\{0,1,2\}$, we have $[s]=[s']$ in $H_1^*(p)$. Thus, there is a 2-chain $\alpha'$ having a $1$-chain $s_0+s_1-s$ as its boundary,
and so $[s]=[s_0]+[s_1]$.
\end{proof}

Now, we summarize some properties of endpoint pairs of $1$-shells which follow from
Theorem \ref{endpt_class} and Theorem \ref{endpt_gpstr}. We define an equivalence relation $\sim$ on the set of pairs of realizations $p$ as follows: For $a,a',b,b'\models p$, $(a,b)\sim(a',b')$ if two pairs $(a,b)$ and $(a',b')$ are endpoint pairs of $1$-shells $s$ and $s'$ respectively such that $[s]=[s']\in H^*_1(p)$. We write $\CE^*=p(\CM)\times p(\CM)/\sim$. We denote the class of $(a,b)\in p(\CM)\times p(\CM)$  by $[a,b]$. By \ref{endpt_class}, if $ab\equiv a'b'$, then $[a,b]=[a',b']$.
Now, define a binary operation $+_{\CE^*}$ on $\CE^*$ as follows: For
$[a,b],[b',c']\in \CE^*$, $[a,b]+_{\CE^*} [b',c']=[a,c]$, where $bc\equiv b'c'$. Due to Theorem \ref{endpt_gpstr}, this operation is well-defined.

\begin{remark}\label{summary_H1_representation}
The pair $(\CE^*,+_{\CE^*})$ forms a commutative group which is isomorphic to $H^*_1(p)$.
More specifically, for $a,b,c\models p$ and $\sigma\in \aut(\CM)$, we have:
\begin{itemize}
	\item $[a,b]+[b,c]=[a,c]$;
	\item $[a,a]$ is the identity element;
	\item $-[a,b]=[b,a]$;
	\item $\sigma([a,b]):=[\sigma(a),\sigma(b)]=[a,b]$; and
	\item $f:\CE^*\to H^*_1(p)$ sending $[a,b]\mapsto [s]$, where $(a,b)$ is an endpoint pair of $s$, is a group isomorphism.
\end{itemize}
\end{remark}
\noindent {\bf From now on,  we identify $\CE^*$ and $H^*_1(p)$.} Notice that, indeed,  the group structure of $\CE^*$  depends only on the types of $(a,b)$'s with $[a,b]\in \CE^*$.
Hence one may similarly define an equivalence relation on
$$\{ q(x,y): \ p(x)\cup p(y)\subseteq q(x,y) \mbox{ a complete type over } \emptyset\}$$
to form $ \CE^*_{\tp}$, and give a
corresponding group operation to conclude that $\CE^*$ and $\CE^*_{\tp}$ are isomorphic.

\medskip

 Due to the same proof in \cite{KKL},    we can restate  Fact \ref{h1=0}  as follows using the endpoint notion:
\begin{fact}\label{Lascar_zero}
Let $a,b\models p$ be such that $a\equiv^L b$. Then any $1$-shell having $(a,b)$ as its
endpoint pair is the boundary of a $2$-chain; i.e., $[a,b]=0$ in $H_1^*(p)$.
\end{fact}
\subsection{The Lascar group and the first homology groups}\label{lascargptoh1}
Here, using the notion of an ordered bracket, for each $a\models p$ we define a map $\varphi^*_a$ from the automorphism group over $B(=\emptyset)$ into the first homology group of $p$ as follows:
 For $\sigma\in \aut(\CM)$, we let $\varphi^*_a(\sigma)=[a,\sigma(a)]$.
 This map will be proven to be a surjective homomorphism not depending on
  the choice of $a\models p$. Thus, we get a canonical epimorphism
 from $\aut(\CM)$ onto $H_1^*(p)$ and we study its kernel.

\begin{theorem}\label{canonical_epi}
\be
	\item Each $\varphi^*_a$ is an epimorphism.
	\item For $a,b\models p$, $\varphi^*_a=\varphi^*_b$. So we get a canonical map $\varphi^*_p$ from $\aut(\CM)$ into $H_1^*(p)$.
	\item There is a canonical epimorphism $\varPhi^*_p$ from $\gal_L(\CM)$ onto $H_1^*(p)$.
\ee
\end{theorem}
\begin{proof}
(1) Fix $a\models p$. At first, surjectivity of $\varphi^*_a$ comes from the fact that for $b\models p$, there is a $\sigma\in \aut(\CM)$ such that $\sigma(a)=b$. It is enough to show that $\varphi_a$ is a homomorphism. For $\sigma,\tau\in \aut(\CM)$,
$$\begin{array}{c c l}
\varphi^*_a(\sigma\tau)&=&[a,\sigma\tau(a)]\\
&=&[a,\sigma(a)]+[\sigma(a),\sigma\tau(a)]\\
&=&[a,\sigma(a)]+\sigma[a,\tau(a)]\\
&=&[a,\sigma(a)]+[a,\tau(a)]\\
&=&\varphi^*_a(\sigma)+\varphi^*_a(\tau).
\end{array}$$
So $\varphi^*_a$ is a homomorphism.\\

\smallskip

(2) Choose $a,b\models p$. Then there exists
$\tau\in \aut(\CM)$ such that $b=\tau(a)$. For $\sigma\in\aut(\CM)$,
$$\begin{array}{c c l}
\varphi^*_b(\sigma)&=&\varphi^*_a(\tau^{-1}\sigma\tau)\\
&=&\varphi^*_a(\tau^{-1})+\varphi^*_a(\sigma)+\varphi^*_a(\tau)\\
&=&\varphi^*_a(\sigma).
\end{array}$$
Thus $\varphi^*_a=\varphi^*_b$, and we get a canonical epimorphism\\
$$\varphi^*_p \mbox{ }:\ \aut(\CM)\rightarrow H_1^*(p),$$\\
defined by: $\varphi^*_p:=\varphi^*_a \mbox{ for some }a\models p.$

\smallskip

(3) By Fact \ref{Lascar_zero}, the kernel of $\varphi^*_p$ contains $\autf(\CM)$ and
$\varphi^*_p$ induces a canonical epimorphism $\varPhi^*_p$ from $\gal_L(\CM)$ onto $H_1^*(p)$.
\end{proof}

\section{The relativised Lascar groups}\label{lascar}
In this section, we introduce some candidates for the notion of Lascar group of a strong type $p$, which are intended
to be the Lascar group relativised to $p$.
We begin by presenting several automorphism groups. Let $\Sigma(\x)$ be a partial type over $\emptyset$ (with $\x$ of possibly infinite length, and
realizations of $\Sigma$ need not be algebraically closed.)
Recall that $\autf_B(\CM)$ is the subgroup of $\aut_B(\CM)$ generated by
$$\{ f\in \aut_B(\CM):\ f\in \aut_M(\CM)\mbox{ for some model } (B\subseteq )M\prec\CM\},$$
and
$$\gall(T,B):=\aut_B(\CM)/\autf_B(\CM)$$
(which does not depend on the choice of the monster model $\CM$).

\begin{definition}\label{variations_strongauto}
\be
	\item $\aut(\Sigma(\CM)):=\{\sigma \restriction \Sigma(\CM):\ \sigma\in \aut(\CM)\}$;
	\item $\autf_{\res}(\Sigma(\CM)):=\{\sigma \restriction \Sigma(\CM) :\ \sigma\in \Autf(\CM)\}$;
	\item for a cardinal $\lambda> 0$, $\autf_{\fix}^{\lambda}(\Sigma(\CM)):=$
	$$\{\sigma \restriction \Sigma(\CM) : \ \sigma\in \aut(\CM) \mbox{ such that for any }  a_i\models \Sigma \mbox{ and }\a=(a_i)_{i< \lambda},\  \a\equiv^L \sigma(\a)
	 \};$$ and
	\item $\autf_{\fix}(\Sigma(\CM)):=$
	$$\{\sigma \restriction \Sigma(\CM) :\ \sigma\in \aut(\CM) \mbox{ such that } \a\equiv^L \sigma(\a)\mbox{ where }	\a \mbox{ is some enumeration of }
	\Sigma(\CM) \}.$$
\ee
\end{definition}

It is straightforward to see that each of the groups  $\autf_{\res}(\Sigma(\CM))\leq \autf_{\fix}(\Sigma(\CM))\leq \autf^{\lambda}_{\fix}(\Sigma(\CM))$ is a normal subgroup of $\aut(\Sigma(\CM))$.

\begin{definition}\label{variations_galois}
\be
	\item $\gall^{\res}(\Sigma(\CM)):=\aut(\Sigma(\CM))/\autf_{\res}(\Sigma(\CM))$;
	\item $\gall^{\fix, \lambda}(\Sigma(\CM)):=\aut(\Sigma(\CM))/\autf_{\fix}^{\lambda}(\Sigma(\CM))$; and 	
	\item $\gall^{\fix}(\Sigma(\CM)):=\aut(\Sigma(\CM))/\autf_{\fix}(\Sigma(\CM))$.
\ee
\end{definition}

\begin{remark}\label{unctbl}
We have $\autf_{\fix}(\Sigma(\CM))=\autf_{\fix}^{\omega}(\Sigma(\CM))$. So $\gall^{\fix}(\Sigma(\CM))=\gall^{\fix, \omega}(\Sigma(\CM))$.
\end{remark}
\begin{proof}
 We will show by induction on $\lambda\geq \omega$ that tuples $(a_j)_{j<\lambda},(b_j)_{j<\lambda}$ with $a_j,b_j\models \Sigma$ are
 Lascar-equivalent iff all their corresponding countable subtuples are. The base case is clear.
 Suppose the statement is true for all cardinal numbers smaller than
 $\lambda$, and assume that corresponding countable subtuples of $(a_j)_{j<\lambda}$ and $(b_j)_{j<\lambda}$
 are Lascar-equivalent. By the inductive hypothesis, for every $i<\lambda$ there is $n_i<\omega$
 such that the Lascar distance of $a_{<i}(:=(a_j)_{j<i})$ and $b_{<i}$ is equal to $n_i$.
 If there is  $n<\omega$ such that $\{i\in \lambda:n=n_i\}$ is cofinal in $\lambda$, then
 the Lascar distance of $(a_i)_{i<\lambda}$ and $(b_i)_{i<\lambda}$ is $n$ and we are done.
 So let us assume it is not the case, and hence there are $(i_k<\lambda)_{k<\omega}$ such that $n_{i_k}\geq k$.
 Then by compactness, for each $k<\omega$,
 there is a finite subset $I_k$ of $i_k$ such that the Lascar distance of $a_{I_k}:=(a_j)_{j\in I_k}$ and $b_{I_k}$
 is at least $k$. Considering the countable set $I:=\bigcup_{k<\omega}I_k$, we get that $a_I$ and $b_I$ are not Lascar equivalent,
 a contradiction.
\end{proof}

\begin{remark}\label{endowtop}
\be\item
 In \cite{Z} or  \cite{K}, how to endow $\gall(T)$ with
 a canonical  topology   to make it a topological group  is explained as follows.
 For fixed small submodels $M$ and $N$ of $\CM$,  it  easily follows that if 
 $f(M)\equiv_N g(M)$ for $f,g\in \aut(\CM)$, then
 $fg^{-1}\in \autf(\CM)$.
 Hence there are
 the canonical maps $\mu:\Aut(\CM)\to S_M(N)$ (where $S_M(N)$ denotes the Stone space
 of types over $N$ of all conjugates of $M$) mapping  $f$ to $\tp(f(M)/N)$, and $\nu:S_M(N)\to \gall(T)$
such that $\nu\mu:\Aut(\CM)\to \gall(T)$ is the quotient map sending $f$ to $f\autf(\CM)$.
The quotient topology under the map $\nu$ is given to  $\gall(T)$.

\item Analogously, we consider
$\nu ':S_M(N)\to \gall^{\res}(\Sigma(\CM))$ such that $\nu '\mu:\aut(\CM)\to \gall^{\res}(\Sigma(\CM))$ is the quotient map
 sending $f$ to $(f\restriction \Sigma(\CM))\autf_{\res}(\Sigma(\CM))$.  Again, we put on $\gall^{\res}(\Sigma(\CM))$ the quotient topology with
respect to $\nu '$. Notice that $\nu '=\xi\nu$,
where $\xi :\gall(T)\to \gall^{\res}(\Sigma(\CM))$ is given by $\xi(h\autf(\CM))=(h\restriction \Sigma(\CM))\autf_{\res}(\Sigma(\CM))$ (it
 is easy to see that $\xi$ is well-defined).

\item The topology on $\gall^{\res}(\Sigma(\CM))$ defined above is the same as the quotient topology  induced from $\gall(T)$ by $\xi$. In particular, the topology on
 $\gall^{\res}(\Sigma(\CM))$  does not depend on the choice of models $M$ and $N$ above, and $\xi$ is a continuous map: For a subset $W$ of $\gall^{\res}(\Sigma(\CM))$,
  we have that $W$ is open iff $\nu'^{-1}[W]=\nu^{-1}[\xi^{-1}[W]]$ is open in $S_M(N)$ iff
 $\xi^{-1}[W]$ is open in $\gall(T)$.

\item
In a similar manner, by considering the map $\nu'':S_M(N)\to \gall^{\fix}(\Sigma(\CM))$ such that
$\nu''\mu:\aut(\CM)\to \gall^{\fix}(\Sigma(\CM))$ is the quotient map, we equip the group $\gall^{\fix}(\Sigma(\CM))$ with the topology induced
by $\nu''$, which coincides with the topology induced on $\gall^{\fix}(\Sigma(\CM))$ by the quotient
map $\gall(T)\to \gall^{\fix}(\Sigma(\CM))$.
Since the quotient of a topological group by a normal subgroup is
always a topological group with respect to the quotient topology, using the fact that $\gall(T)$ is a topological
group (\cite[Theorem 16]{Z}) we obtain the following corollary.
\ee
\end{remark}

\begin{corollary}
With the topologies defined above, $\gall^{\res}(\Sigma(\CM))$ and $\gall^{\fix}(\Sigma(\CM))$ are topological groups.
\end{corollary}

\begin{proposition}\label{indep_gal_fix}
The group $\gall^{\fix,\lambda}(\Sigma(\CM))$ (for $\lambda\leq \omega$, so in particular $\gall^{\fix}(\Sigma(\CM))$)  does not depend on the choice of the monster model $\CM$.
\end{proposition}
\begin{proof}
Consider two monster models $\CM\prec\CM'$, such that $\CM'$ is $|\CM|^+$-saturated and $|\CM|$-strongly homogeneous.
We define a map $$\eta:\gall^{\fix,\lambda}(\Sigma(\CM))\to \gall^{\fix,\lambda}(\Sigma(\CM'))$$ by $$\eta([f\restriction \Sigma(\CM)])=[f'\restriction \Sigma(\CM')],$$
where $f'\in \aut(\CM')$ is any extension of $f\in \aut(\CM)$. Let us check that $\eta$ is well-defined. Suppose
that two automorphisms $f_1,f_2\in \aut(\CM)$ determine the same element in $\gall^{\fix,\lambda}(\Sigma(\CM))$; i.e. $g:=(f_1f_2^{-1})\restriction \Sigma(\CM)$
belongs to $\autf^\lambda_{\fix}(\Sigma(\CM))$. Take any $f_1',f_2'\in\aut(\CM')$ extending $f_1$ and $f_2$ respectively. To see
that $g':=(f_1'f_2'^{-1})\restriction \Sigma(\CM')$ is in $\autf^\lambda_{\fix}(\Sigma(\CM'))$, take any $\lambda$-tuple
$a'$ of elements of  $\Sigma(\CM')$. Pick $a\in \CM$ which is Lascar equivalent to $a'$. Then $g'(a')\equiv^L g'(a)=g(a)\equiv^L a\equiv^L a'$,
which shows that $g'\in \autf_{\fix}^\lambda(\CM')$ (by Remark \ref{unctbl}), and so $\eta$ is well-defined.

Now it is clear that $\eta$ is an injective homomorphism. To see that it is onto, consider
any element $[g\restriction \Sigma(\CM')]$ of
$\gall^{\fix,\lambda}(\Sigma(\CM'))$, where $g\in \aut(\CM')$. By the  argument in Remark
\ref{endowtop}(1), we can find $g'\in \aut(\CM')$ such that
$gg'^{-1}\in \Autf(\CM')$ and $g'[\CM]=\CM$. Then $\eta([g'\restriction \Sigma(\CM)])=[g'\restriction \Sigma(\CM')]=[g\restriction \Sigma(\CM')]$.
This shows that $\eta$ is an isomorphism.
\end{proof}
\begin{notation} Due to above Proposition \ref{indep_gal_fix},
we write $\gall^{\fix, n}(\Sigma), \gall^{\fix}(\Sigma)$ for the groups $\gall^{\fix, n}(\Sigma(\CM)), \gall^{\fix}(\Sigma(\CM))$, respectively.
\end{notation}

\begin{question}\label{indep_gal_type}
Is the group $\gall^{\res}(\Sigma(\CM))$ independent from the choice of the monster model $\CM$?
What is an example in which $\autf_{\res}(\Sigma(\CM))$ differs from $\autf_{\fix}(\Sigma(\CM))$?
\end{question}

\begin{remark}\label{modelsol.}
If there are realizations  $a_i\in \Sigma(\CM)$ and   a small submodel $(B\subseteq )M$ of $\CM$ such that $M\subseteq \dcl(B a_i|\ i<\lambda)$,
then $\autf_{\res}(\Sigma(\CM))=\autf_{\fix}(\Sigma(\CM)).$ In particular, if $M\subseteq \dcl(Ba_0)$, then 
$\autf_{\res}(\Sigma(\CM))=\autf_{\fix}(\Sigma(\CM))= \autf_{\fix}^1(\Sigma(\CM)).$
\end{remark}

\begin{fact}\label{canonical_epimorphism}
Recall from Section \ref{lascargptoh1} that we have a canonical epimorphism $$\psi^*_{p}:\aut(p(\CM))\rightarrow H^*_1(p)$$
 sending each $\sigma\in\aut(p(\CM))$ to $[a,\sigma(a)]$ for some/any realization $a$ of $p$.
  Due to Fact \ref{Lascar_zero}, $\ker( \psi^*_{p})$ contains $\autf_{\fix}(p(\CM))$. Hence,
  this induces
   a canonical epimorphism $\Psi^*_{p}:\ \gall^{\fix}(p)\rightarrow H^*_1(p)$ as well.
\end{fact}

\begin{remark}\label{kernel_canonical_epimorphism}
Note that   $\aut(p(\CM))/\ker(\psi^*_p)$ is  isomorphic to $H^*_1(p)$, which is  independent from the choice of the monster model.
Since $H^*_1(p)$ is abelian, $\ker(\psi^*_p)$ contains the derived subgroup of $\aut(p(\CM))$.
We shall figure out what $\ker(\psi^*_p)$ is, and it will turn out that even the kernel (so $H^*_1(p)$ too) is 
independent from the choice of $\indo^*$.
\end{remark}

\section{Characterization of the first homology groups}\label{characterization}

The goal of this section is to identify what $\ker(\psi^*_p)$ is.
In \cite{KKL},\cite{KL},    the $2$-chains in $p$ (in the sense of thorn-independence)
with $1$-shell boundaries
are classified when $T$ is rosy.
However, again, the only properties
of  thorn-forking used there are {\em the basic 5 axioms}: finite character, normality, symmetry, transitivity, and extension. Therefore, the same conclusion
can be  obtained in our context of $*$-independence in any $T$. 

In particular we obtain the following from \cite[3.14]{KKL}:

\begin{remark}\label{chain-walk}
Let $s=f_{01}+f_{12}-f_{02}$ be a $1$-$*$-shell with $\supp(f_{ij})=\{i,j\}$. Then $s$ is the boundary of some 2-$*$-chain in $p$ iff $s$ is the boundary of some 2-$*$-chain $$\alpha=\sum_{i=0}^{2n}\limits (-1)^i a_i$$ with 2-$*$-simplicies $a_i$, which is a {\em chain-walk} from $f_{01}$ to $f_{12}$. We call the 2-$*$-chain $\alpha$ a {\em chain-walk} from $f_{01}$ to $f_{12}$ if, 
\be\item there are non-zero numbers $k_0,\ldots,k_{2n+1}$ (not necessarily distinct) such that   $k_0=k_{2n}=1$, $k_{2n+1}=2$, and for $ i\leq 2n$, $\supp(a_i)=\{0,k_i,k_{i+1}\}$;
        \item  $\Bd^2  a_0 = f_{01}$, $\Bd^2  a_{2n}= f_{12}$; and
        \item for $0 \le i < n$,
        $$\Bd^0 a_{2i}= \Bd^0 a_{2i+1}, \  \Bd^2  a_{2i+1}=\Bd^2 a_{2i+2}.$$
        \ee
\end{remark}    
\noindent Note that actually in \cite[3.14]{KKL}, it is given as a chain-walk from $f_{01}$ to $-f_{02}$ but the same proof gives a chain-walk from $f_{01}$ to $f_{12}$.
 
Now due to the fact that $\Bd(\alpha)=s$ and $\alpha$ is a chain-walk, we can directly obtain the following fact.

%
%
%

\begin{theorem}\label{improved_classification}
A $1$-$*$-shell $s$  in $p$ is the boundary of a $2$-chain if and only if there is a representation $(a,b,c,a')$ of $s$ such that
for some $n\geq 0$ there is a finite sequence $(d_i)_{0\le i\le 2n+2}$ of realizations of $p$ satisfying the following conditions:
\be
	\item $d_0=a$, $d_{2n+1}=c$ and $d_{2n+2}=a'$;
	\item $\{d_{j},d_{j+1},b\}$ is $*$-independent for each $0\le j\le 2n+1$; and
	\item there is a bijection $$\sigma:\{0,1,\ldots,n\}\to\{0,1,\ldots,n\}$$ such that
	$d_{2i} d_{2i+1}\equiv d_{2\sigma(i)+2}d_{2\sigma(i)+1}$ for $0\le i \le n$.
\ee
\end{theorem}
\begin{proof}
Let $s$ be a $1$-shell with $\supp(s)=\{0,1,2\}$ in $p$, which  is the boundary of a $2$-chain.
Then, by Remark \ref{chain-walk}, we have a 2-$*$-chain-walk $\sum_{i=0}^{2n}\limits (-1)^i a_i$ from $f_{01}$ to $f_{12}$ with the boundary $s$. Then there are $d_0,d_1,\ldots,d_{2n+1},b\models p$ such that
\be

	\item $\{d_{j},d_{j+1},b\}$ is $*$-independent for each $0\le j\le 2n$;
	\item $a_{2i}(\{0,1,2\})\equiv [d_{2i}bd_{2i+1}]_{\{0,1,2\}}$ and $a_{2i+1}(\{0,1,2,\})\equiv [d_{2i+2}bd_{2i+1}]_{\{0,1,2\}}$ for each $0\le i\le n-1$.
	\item $f_{02}(\{0,2 \})\equiv [d_{2i_0}d_{2i_0+1}]_{\{0,2\}}$ for some $0\le i_0\le n-1$.

\ee

\noindent Take a tuple $(d_0,b,d_{2n+1},d_{2i_0})$, which is a representation of $s$. Since $\alpha$ is a 2-$*$-chain-walk, there is a bijection $$\sigma:\{0,1,\cdots,n\}\smallsetminus \{i_0\}\to\{0,1,\cdots,n-1\}$$ such that $d_{2i} d_{2i+1}\equiv d_{2\sigma(i)+2}d_{2\sigma(i)+1}$ for $0\le i\neq i_0 \le n$. Set $\sigma':=\sigma\cup \{(i_0,n)\}$ and $d_{2n+2}:=d_{2i_0}$. We get a desired result from the bijection $\sigma'$ and the sequence $(d_0,d_1,\ldots,d_{2n+1},d_{2n+2})$.
%

Conversely, we assume that there are a representation $(a,b,c,a')$, a finite sequence $(d_i)_{0\le i\le 2n+2}$ of realizations of
$p$, and a bijection $\sigma$ on $\{0,1,\ldots, n\}$ for some $n\geq 0$ such that
\be
	\item $d_0=a$, $d_{2n+1}=c$ and $d_{2n+2}=a'$;
	\item $\{d_{j},d_{j+1},b\}$ is $*$-independent for each $0\le j\le 2n+1$; and
	\item $d_{2i} d_{2i+1}\equiv d_{2\sigma(i)+2}d_{2\sigma(i)+1}$ for $0\le i \le n$.
\ee
Put $i_0=\sigma^{-1}(n)$. From the subsequence $(d_i)_{0\le i\le 2n+1}$, we get a 2-$*$-chain-walk $\alpha=\sum_{i=0}^{2n}\limits (-1)^i a_i$ from $f_{01}$ to $f_{12}$ with the boundary $f_{01}-a_{2i_0}\restriction \{0,2\}+f_{12}$. Since $d_{2i_0}d_{2i_0+1}\equiv d_{2n+2}d_{2n+1}=a'c$, we can make $a_{2i_0}\restriction \{0,2\}=f_{02}$ so  that $\Bd \alpha=s$. 

\end{proof}

\begin{corollary}\label{endpt_classification}
For $a,a'\models p$, $[a,a']=0$ in $H_1^*(p)$ if and only if for some $n\geq 0$ there is a finite sequence $(d_i)_{0\le i\le 2n+2}$ of realizations of $p$ satisfying the following conditions:
\be
	\item $d_0=a$, and $d_{2n+2}=a'$;
	\item $\{d_{j},d_{j+1}\}$ is $*$-independent for each $0\le j\le 2n+1$; and
	\item there is a bijection $$\sigma:\{0,1,\ldots,n\}\to\{0,1,\ldots,n\}$$ such that
	$d_{2i} d_{2i+1}\equiv d_{2\sigma(i)+2}d_{2\sigma(i)+1}$ for $0\le i \le n$.
\ee
\end{corollary}
\begin{proof}
Fix $a,a'\models p$. The left-to-right direction is clear from Theorem \ref{improved_classification}. For the right-to-left direction, we assume that there is a finite sequence $(d_i)_{0\le i\le 2n+2}$ of realizations of $p$ satisfying the following conditions:
\be
	\item $d_0=a$, and $d_{2n+2}=a'$;
	\item $\{d_{j},d_{j+1}\}$ is $*$-independent for each $0\le j\le 2n+1$; and
	\item there is a bijection $$\sigma:\{0,1,\ldots,n\}\to\{0,1,\ldots,n\}$$ such that
	$d_{2i} d_{2i+1}\equiv d_{2\sigma(i)+2}d_{2\sigma(i)+1}$ for $0\le i \le n$.
\ee
Take $b\models p$ such that $b\indo^* d_0d_1\ldots d_{2n+2}$. Then the tuple $(a,b,d_{2n+1},a')$ and the sequence $(d_i)_{0\le i\le 2n+2}$ gives a 1-$*$-shell which is a boundary of a 2-$*$-chain by Theorem \ref{improved_classification}.
\end{proof}
By now, as promised, we can identify $\ker(\psi^*_p)$.

\begin{Theorem}\label{kernel_derived}
For each $h\in K:=\ker(\psi^*_p)$ and $a\models p$, there is an automorphism $h'$  in the derived subgroup
$G'$ of $G:=\aut(p(\CM))$ such that $h(a)=h'(a)$. Thus,
$K(\geq G')$ is the normal subgroup of $G$ consisting of all automorphisms fixing all
orbits of elements of $p(\CM)$ under the action of $G'$, and $H^*_1(p)=G/K$.
\end{Theorem}
\begin{proof}
If the first statement is true then the second statement clearly  follows  since
$G'\leq K=\ker(\psi^*_p)$
(as $G/K\cong H^*_1(p)$ is abelian). So let us prove the first statement.

Let $h\in K$ and $a\models p$. Thus $[a,h(a)]=0$ (in $\CE^*=H^*_1(p)$). By Corollary \ref{endpt_classification}, there are an integer $n\ge 0$ and a finite sequence $(d_i)_{0\le i\le 2n+2}$ of realizations of $p$ such that
\be
	\item $d_0=a$, and $d_{2n+2}=h(a)$;
	\item $\{d_{j},d_{j+1}\}$ is $*$-independent for each $0\le j\le 2n+1$; and
	\item there is a bijection $$\sigma:\{0,1,\ldots,n\}\to\{0,1,\ldots,n\}$$ such that
	$d_{2i} d_{2i+1}\equiv d_{2\sigma(i)+2}d_{2\sigma(i)+1}$ for $0\le i \le n$.
\ee

By (3), we have the following automorphisms in $G$:
\begin{itemize}
	\item $\eta_i:d_{2i}\mapsto d_{2i+1}$ for $0\le i \le n$;
	\item $\eta'_j:d_{2j+1}\mapsto d_{2j+2}$ for $0\le j\le n$; and
	\item $f_i:d_{2i}d_{2i+1}\mapsto d_{2\sigma(i)+2}d_{2\sigma(i)+1}$ for $0\le i \le n$.
\end{itemize}
Thus for $0\leq i\leq n$,  $\eta_{i}\eta'_{i-1}\eta_{i-1}\cdots\eta'_0\eta_0(d_0)=d_{2i+1}$, $\eta'_{i}\eta_{i}\cdots\eta'_0\eta_0(d_0)=d_{2i+2}$,
 and $$\eta'_n\eta_{n}\eta'_{n-1}\eta_{n-1}\cdots\eta'_0\eta_0(d_0)=d_{2n+2}=h(a) \ \ (\dagger).$$ Moreover, $\eta'_{\sigma(i)}(d_{2\sigma(i)+1})=f_i\eta_i^{-1} f_i^{-1}(d_{2\sigma(i)+1})$ for $0\le i\le n$, so  $\eta'_{i}(d_{2i+1})=f_{\sigma^{-1}(i)}\eta_{\sigma^{-1}(i)}^{-1} f_{\sigma^{-1}(i)}^{-1}(d_{2i+1})$ \ \ ($\ddagger$).
 From $(\dagger)$ and $(\ddagger)$, we get that $g_0(a)=h(a)$, where
  $$g_0:=(f_{\sigma^{-1}(n)}\eta_{\sigma^{-1}(n)}^{-1}f_{\sigma^{-1}(n)}^{-1})\eta_n\cdots (f_{\sigma^{-1}(0)}\eta_{\sigma^{-1}(0)}^{-1}f_{\sigma^{-1}(0)}^{-1})\eta_0.$$
We claim that $g_0\in G'$: Since $G/G'$ is abelian (so $gkG'=kgG'$), we get that $g_0 G'=g_1 G'$, where
$$g_1:=\eta_{\sigma^{-1}(n)}^{-1}\eta_n\cdots \eta_{\sigma^{-1}(0)}^{-1}\eta_0. $$
Moreover, since $\sigma$ is a permutation of $\{0,1,\ldots,n\}$, it follows
(using again that $G/G'$ is abelian) that $g_1 G'=1_G G'=G'$, so  $g_1,g_0\in G'$.

\end{proof}


\begin{remark}
Due to above Theorem \ref{kernel_derived}, $H^*_1(p)$, which of course does not depend on the
choice of a monster model, is always the same regardless of our choice of independence
$\indo^*$ satisfying the 5 basic axioms. 
Hence, we can denote it simply by $H_1(p)$. Moreover, $H_1(p)$ can also be considered as a quotient group of
$\gall(T)$ or $\gall^{\fix}(p)$,  which equivalently endow $H_1(p)$ with a topological group structure. 
%
\end{remark}

\begin{remark}\label{h1_stp_model}
If $p$ is a strong type of a model, then for $G:=\gall(T)$,  $G\cong \gall^{\fix,1}(p)=\gall^{\fix}(p)$ and $H_1(p)\cong G/G'$ as topological groups.
\end{remark}
\begin{proof}
Suppose $p$ is a strong type of a small submodel $M$ of $\CM$. By Remark \ref{modelsol.}, $\autf^1_{\fix}(p(\CM))=\autf_{\fix}(p(\CM))=\autf_{\res}(p(\CM))$. 
It remains to show that $\gall^{\fix,1}(p(\CM))\cong \gall(T)$. Consider the
projection $\pi_{\fix,1} : \aut(\CM)\rightarrow \gall^{\fix,1}(p(\CM))$ sending $\sigma$ 
to $(\sigma \restriction p(\CM))\autf_{\fix}^1( p(\CM) )$. Suppose $\sigma$ is in the kernel 
of $\pi_{\fix,1}$. Then $\sigma(M)\equiv^L M$, and there is $\tau\in \autf(\CM)$ such that 
$\sigma\restriction M=\tau \restriction M$. Thus $\tau^{-1}\circ \sigma \restriction M=\id_M$. 
So $\tau^{-1}\circ \sigma \in \autf(\CM)$ and $\sigma \in \autf(\CM)$. Therefore, the kernel of
$\pi_{\fix,1}$ is a subgroup of $\autf(\CM)$. Also, it is easy to check that $\autf(\CM)$ is a
subgroup of the kernel of $\pi_{\fix, 1}$. Thus we have that
$\gall(T)\cong \gall^{\fix,1}(p(\CM))$, witnessed by the 
isomorphism induced from $\pi_{\fix,1}$ (actually this is an isomorphism of topological groups).

Now let $\varphi_p : \aut(\CM)\rightarrow H_1(p)$ be the epimorphism sending
$\sigma\in \aut(\CM)$ to $[M,\sigma(M)]$, as defined in \ref{canonical_epi}. It is enough to show that the kernel of
$\varphi_p$ is generated by the automorphisms in $\aut(\CM)'\cup \autf(\CM)$, so
$\ker(\varphi_p)=\aut(\CM)'\autf(\CM)$. It is clear from Theorem \ref{kernel_derived} that
$\aut(\CM)'\autf(\CM)\subseteq \ker(\varphi_p)$. Conversely, 
suppose $\sigma\in \ker(\varphi_p)$. By Theorem \ref{kernel_derived}, there is
$\tau\in\aut(\CM)'$ such that $\sigma\restriction M=\tau \restriction M$.
So $(\tau^{-1}\circ\sigma) \restriction M=\id_M$ and $\tau^{-1}\circ\sigma\in \autf(\CM)$.
Thus, $\sigma \in \tau\autf(\CM)\subseteq \aut(\CM)'\autf(\CM)$. Therefore $H_1(p)\cong G/G'$,
and furthermore they are homeomorphic as topological groups.
\end{proof}

As a corollary to Theorem \ref{kernel_derived}, we get the following characterization
of the equality of strong types and Lascar strong types, {\em for any theory}:
\begin{corollary}\label{equality_stp_Lstp} The following conditions are equivalent :
\begin{enumerate}
	\item $p$, a strong type  of an algebraically closed tuple, is a Lascar strong type;
	\item $\gall^{\fix,1}(p)$ is abelian and $H_1(p)=0$; and
	\item Both $\gall^{\fix,1}(p)$ and $H_1(p)$ are trivial.
\end{enumerate}
In particular, if $\gall^{\fix}(p)$ is abelian, then $p$ is a Lascar strong type if and only if $H_1(p)=0$.
Moreover, if $p$ is a strong type of a model, then $p$ is a Lascar strong type if and only if $\gall(T)$ is abelian
and $H_1(p)=0$.
\end{corollary}
\begin{proof}
The implication from (3) to (2) is trivial. The implications from (1) to (2) and (3) are easy. Suppose $p$ is a Lascar
strong type. Then, by the definition, $\gall^{\fix,1}(p)=0$ and $H_1(p)=0$.  It is enough to show (2) $\Rightarrow$ (1).
Suppose $\gall^{\fix,1}(p)$ is abelian and $H_1(p)=0$. It is enough to show that for any $a\in p(\CM)$ and any
$\sigma \in \aut(p(\CM))$, $a\equiv^L \sigma(a)$. Choose $a\in p(\CM)$ and $\sigma\in \aut(p(\CM))$ arbitrarily.
Since $H_1(p)=0$, there is $\tau\in \aut(p(\CM))'$ such that $\sigma(a)=\tau(a)$ by Theorem \ref{kernel_derived}.
The derived group of $\aut(p(\CM))$ is a subgroup of $\autf_{\fix}^1(p(\CM))$, because
$\gall^{\fix,1}(p)$ is abelian.
Therefore, we have $a\equiv^L \sigma(a)(=\tau(a))$.

The rest comes from Remark \ref{h1_stp_model}. 
\end{proof}

Next, we consider the orbit equivalence relation $\equiv^{H_1}$ on
$p(\CM)$ under the action of  $K$ (equivalently $G'$) in Theorem
\ref{kernel_derived} (i.e., for $a,b\models p$, $a\equiv^{H_1} b$ iff there is
$f\in K$ (or $\in G'$) such that $b=f(a)$ iff $[a,b]=0\in H_1(p)$). We show now 
that this equivalence relation is an $F_{\sigma}$-relation (in any theory); i.e., there are countably many  $B$-type-definable reflexive,
 symmetric relations $R_i(x,y)$ such that $$p(x)\wedge p(y)\models  x \equiv^{H_1} y\leftrightarrow \bigvee_{i<\omega} R_i(x,y):$$
Consider an invariant, symmetric, reflexive relation $R$ such that for
$\a,\b\in \CM$, $R(\a,\b)$ if and only if there are 
$\sigma,\tau\in\aut(\CM)$ such that $\b=[\sigma,\tau](\a)$, where 
$[\sigma,\tau]:=\sigma^{-1}\tau^{-1}\sigma\tau$, if and only if there are $c_1$, $c_2$, and $c_3$ such that $ac_3\equiv c_1 c_2$ and $c_2 c_3\equiv c_1 b$. Define $$R_i(\x,\y) \equiv \begin{cases}
\x=\y&\mbox{if}\ i=0\\
\overbrace{R\circ \cdots \circ R}^{i}(\x,\y)&\mbox{if}\ i\ge 1.
\end{cases}
$$
\noindent Then by Theorem \ref{kernel_derived}, $$p(x)\wedge p(y)\models  x \equiv^{H_1} y\leftrightarrow \bigvee_{i<\omega} R_i(x,y).$$
Next, define the {\em $H_1$-distance on} $p$ as follows: For $a,b\models p$,
$$
d_{H_1}(a,b):=
\begin{cases}
\min\{n|R_n(a,b)\} & \mbox{if } a\equiv^{H_1} b\\
\infty & \mbox{otherwise,}
\end{cases}
$$
and the $H_1$-diameter on $p$ by: $$d_{H_1}(p):=\max\{d_R(a,b)|\ a,b\models p\}.$$
Applying the Newelski's result from \cite{N} on the possible cardinality of the set of
classes of bounded invariant equivalence relations on a type,
we know that the cardinality of $H_1^*(p)$ is at least $2^{\aleph_0}$ if the equivalence
relation $\equiv^{H_1}$ on $p$ is not type-definable, and in the other case,
we also have that the possible cardinality of $H_1$ is one or at
least $2^{\aleph_0}$ by Appendix A.
\begin{Theorem}\label{cardinality_h1} For any theory $T$:
\begin{enumerate}
	\item The equivalence relation $\equiv^{H_1}$ on $p$ is type-definable if and only if $d_{H_1}(p)$ is finite.
	\item The cardinality of  $H_1(p)$ is one or $\ge 2^{\aleph_0}$.
\end{enumerate}

\end{Theorem}

We finish this section by posing the following question for simple theories. 

\begin{Question}\label{simpleh_1} In a simple theory, is the first homology group of a strong type always trivial?
\end{Question}

\section{Examples}


\subsection{Topological groups and the first homology groups of types}
In this subsection, we argue that all connected abelian compact groups can occur as the first homology groups of strong types 
(Here, compact topological spaces are Hausdorff by definition). At first, note that in \cite{Z}   M. Ziegler showed
(using a result of E. Bouscaren, D. Lascar, and A. Pillay) that any compact group occurs as the Lascar Galois group of a complete theory.
\begin{fact}\label{cpt_gp_Lascar_gp}\cite{Z}
Let $G$ be a compact group. Then there is a complete theory $T_G$ whose Lascar Galois group is isomorphic to $G$.
\end{fact}
\noindent From Remark \ref{h1_stp_model}, we also know that the first homology group of a strong type of a model is isomorphic to the abelianization of the connected component of Lascar Galois group. 
So, if we take $G$ in Fact \ref{cpt_gp_Lascar_gp} as an abelian and connected group, we conclude that the first homology group of a strong type of model in $T_G$ is isomorphic to the Lascar Galois group $G$ itself.
\begin{theorem}\label{ab_conn_gp_h1}
For each abelian connected compact group $G$, there is a strong type of a model of a complete theory whose first homology group is isomorphic to $G$.
\end{theorem}

\begin{remark}\label{h1=0_not_enough} There is a strong type $p$ in a theory with trivial first homology group, which
is not a Lascar strong type.   In other words, 
in Corollary \ref{equality_stp_Lstp} (2), we cannot omit the condition of abelianness of $\gall^{\fix}(p)$ to conclude that a 
given strong type $p$ is a Lascar strong type. If $G$ is a non-trivial connected compact group whose commutator subgroup is
itself, then the first homology group of a strong type of a model of $T_G$ is trivial. In this case, the strong type is not 
a Lascar strong type because the Lascar Galois group is not trivial. For example, we can take $G:=SU(3)$ as such a group.

\end{remark}

\subsection{Some computation of the first homology group of a type} Here we give a more concrete
example of a strong type in a  rosy theory with a non-trivial first homology group. In \cite{KKL}, S. Kim, and the second and third authors considered the structures
$\CM_{n}=(M;S;g_{1/n})$ (which were earlier studied in \cite{CLPZ})
for each $n\in \BN\setminus \{0\}$, where
\be
	\item $M$ is a saturated circle;
	\item $g_{1/n}$ is the clockwise rotation by $2\pi/n$ radians; and	
	\item $S$ is a ternary relation such that $S(a,b,c)$ holds if $a,b,c$ are distinct and $b$ comes before $c$ going
	around the circle clockwise starting at $a$,
	
\ee
and it was shown that the unique strong 1-type $p_n$ in $S_1(\emptyset)$ has the trivial
first homology group for every $n$, and is actually a Lascar strong type.
Now, we consider a structure $\CM=(M;S;g_{1/n} : n\in \BN\setminus\{0\})$ expanding the
structures $\CM_{n}$ by adding all rotation functions by $2\pi/n$-radians for each
$n\in \BN\setminus\{0\}$ at the same time (when we write $g_r$ for $r=m/n$ in $\BQ\cap[0,1)$, it means $g_{1/n}^m$). 
We show that $\Th(\CM)$ is a rosy theory. In \cite{EO}, C. Ealy and A. Onshuus gave a sufficient condition for a theory to be rosy.

\begin{Fact}\label{character_rosy}
Any theory $T$ which weakly eliminates imaginaries and for which the
algebraic closure defines a pregeometry is rosy of thorn $U$-rank $1$.
\end{Fact}

At first, we show that $\Th(\CM)$ has weak elimination of imaginaries. In \cite{P}, B. Poizat defined a theory $T$ to have 
{\em weak elimination of imaginaries} if every definable set has a smallest algebraically closed set over which it is definable. 
By repeating the argument from \cite{KKL},
we obtain the following sufficient condition for weak
elimination of imaginaries in an $\aleph_0$-categorical theory:
\begin{theorem}\label{wei_omegacat}
Let $T$ be $\aleph_0$-categorical and let $\CM=(M,\ldots)$ be a saturated model of $T$.
Suppose that if $X\subset M^1$ is definable over each of
two algebraically closed sets $A_0$ and $A_1$, then $X$ is definable over $B:=A_0\cap A_1$.


 Then, for any subset $Y$ of $M^n$, if $Y$ is both $A_0$-definable and $A_1$-definable, then
it is $B$-definable. Furthermore, in this case, $T$ has weak elimination of imaginaries.
\end{theorem}
\begin{proof}
Let $A_0=\acl(A_0)$, $A_1=\acl(A_1)$, and $B=A_0\cap A_1$. We use induction on $n$.
If $n=1$, the conclusion holds by assumption. Let us show that the conclusion holds for
$n+1$ assuming it holds for $n$. Put $A_0=\acl(A_0)$, $A_1=\acl(A_1)$, and $B=A_0\cap A_1$.
Since, by $\aleph_0$-categoricity, the algebraic closure of 
a finite set is finite, we may assume that $A_0$ and $A_1$ are finite, and so is $B$. Let $Y\subset M^{n+1}$ be
$A_i$-definable by a formula $\phi_i(x_0,\ldots,x_n;\a_i)$ with $\a_i\subset A_i$
for $i=0,1$. Then, for each $c\in M$, the fiber of $Y$ over $c$,
$Y_c:=\{ \x \in M^n|\ \phi_i(\x,c;\a_i)\}$ is $cB$-definable by induction.
By $\aleph_0$-categoricity, there are only finitely many formulas over $\emptyset$ modulo
$T$, and it easily follows that for each $c\in M^1$, $\phi_i(x_0,\ldots,x_{n-1},c,\a_i)$ is
$B$-definable. Thus, $Y$ is $B$-definable.

Since (again by $\aleph_0$-categoricity) there is no infinite descending chain of algebraically closed sets generated by
finitely many elements, we conclude that any definable set has a smallest algebraically
closed set over which it is definable. Thus, $T$ weakly eliminates imaginaries.
\end{proof}
\noindent As a corollary to Theorem \ref{wei_omegacat}, it was shown in \cite{KKL}
that for each $n\ge 2$, $\Th(\CM_{n})$ has weak elimination of imaginaries.
\begin{fact}[\cite{KKL}]\label{wei_M_1_reduct}
For each $n\ge 2$, $\Th(\CM_{n})$ weakly eliminates imaginaries.
\end{fact}

Next, we will see that the theory of $\CM$ has quantifier-elimination.
\begin{definition}\label{rotation_closure}
Let $M$ be the non-standard circle which is the universe of $\CM$.
For $A\subset M$, let $\cl(A):=\{g_r(a)|\ a\in A,\ r\in \BQ\cap [0,1) \}$. Later, we will see that $\cl(A)=\dcl(A)=\acl(A)$ in the home sort of $\CM$. It is also easy to see that $\cl(A)$ is a substructure of $\CM$.
\end{definition}

\begin{theorem}\label{QE_M_1}
The theory of $\CM$ has quantifier-elimination.
\end{theorem}
\begin{proof}
Take two small subsets $A,B\subset M$ such that $A=\cl(A)$ and $B=\cl(B)$ in $M$,
and a partial isomorphism $f:A\to B$. Take $a\in M\setminus A$. We will find $b\in M\setminus B$
such that the map $f\cup\{(a,b)\}$ can be extended to an embedding from $\cl(Aa)$ to $\cl(Bb)$ in
$\CM$. Then, the quantifier-elimination in $\Th(\CM)$ comes from a standard argument.
We divide $A$ into two parts: $A_0 :=\{x\in A|\ S(a,x,g_{1/2}(a))\}$ and
$A_1 :=\{x\in A|\ S(g_{1/2}(a),x,a)\}$. Then $B$ is also divided into two parts:
$B_0=f(A_0)$ and $B_1=f(A_1)$. Take arbitrary $b\in M$ such that for all
$y_0\in B_0,\ y_1\in B_1$, we have that $S(y_1,b,y_0)$. Then $b$ is a desired element.
\end{proof}

\begin{theorem}\label{wei_M_1}
The theory of $\CM$ weakly eliminates imaginaries, and is rosy of thorn $U$-rank $1$.
\end{theorem}
\begin{proof}
By quantifier elimination, in the structure $\CM$ there is no infinite descending chain of
algebraic closures of
finite sets. It is enough to show that if $X\subset M^n$ is
$A_0(=\acl(A_0))$-definable and $A_1(=\acl(A_1))$-definable, then $X$ is $A_0\cap A_1(=B)$-definable.
(Then $X$ has a smallest algebraically closed set over which it is definable,
and $\Th(\CM)$ has weak
elimination of imaginaries.)

Let $A_i=\acl(A_i)=\cl(A_i)$ for $i=0,1$ and put $B=A_0\cap A_1$. Let $X\subset M^m$ be
$A_i$-definable in $\CM$. Then $X$ is definable over $A_i$ for $i=0,1$ in some reduct
$\CM_{n}$ of $\CM$. Since $\CM_{n}$ weakly eliminates imaginaries, $X$ is definable over $B$
in $\CM_{n}$ by a formula $\psi(\x,\b)$. Then
$X$ is $B$-definable in $\CM$ by the same formula $\psi(\x,\b)$.

By quantifier elimination, it is easily verified that the algebraic closure in $\CM$
gives a trivial pregeometry (i.e. $\acl(A)=\cup_{a\in A}\acl(\{a\})$). Thus, by Fact \ref{character_rosy}, $\Th(\CM)$ is a rosy
theory of thorn $U$-rank $1$.
\end{proof}
\noindent There is only one $1$-strong type over the empty set in $\CM$: $p_0(x)\equiv\{x=x\}$.

In $\CM$, for a fixed $a\in M$, we observe that the types in $S_1(a)$ correspond
to elements of the unit circle, where the points with rational spherical coordinates are
tripled. Using this observation, we compute the first homology group of
$p_0$ in $\CM$:
\begin{theorem}\label{h1_in_M_1}
In $\CM$, the first homology group of $p_0$ is isomorphic to $\BR/\BZ$.
\end{theorem}
We start with defining a distance-like notion between two points on $M$. We fix an infinitesimal $\epsilon$. For a subset $Y\subset \BQ$, we define $Y^*:=Y\cup\{y\pm\epsilon| y\in Y \}$. We write $X_{\BQ}$ for $X\cap \BQ$ for a subset $X$ in $\BR$.
\begin{definition}\label{distance}
Let $a,b\in M$ be two elements. We define \em{the $S$-distance from $a$ to $b$},
denoted by $\sdist(a,b)$. The $S$-distance has values in $[0,1)\cup[0,1)^*_{\BQ}\cup\{1-\epsilon\}$. Let $r\in (0,1)_{\BQ}$ and $r'\in (0,1)\setminus \BQ$.
\be
	\item $\sdist(a,b)=0$ if $b=a$;
	\item $\sdist(a,b)=\epsilon$ if for all $s\in (0,1)_{\BQ}$, $\CM\models S(a,b,g_{s}(a))$;
	\item $\sdist(a,b)=1-\epsilon$ if for 
all $s\in (0,1)_{\BQ}$, $\CM\models S(g_{s}(a),b,a)$;
	\item $\sdist(a,b)=r$ if $b=g_{r}(a)$;
	\item $\sdist(a,b)=r-\epsilon$ if for $s\in (0,1)_{\BQ}$ with $s<r$, $\CM \models S(g_s(a),b,g_{r}(a))$;
	\item $\sdist(a,b)=r+\epsilon$ if for $t\in (0,1)_{\BQ}$ with $r< t$, $\CM \models S(g_{r}(a),b,g_{t}(a))$;

	\item $\sdist(a,b)=r'$ if for $s<t\in [0,1)_{\BQ}$ such that $s<r'<t$, $\CM \models S(g_s(a),b,g_t(a))$.	
\ee
\end{definition}
\noindent In Appendix B, using Dedekind cuts, we develop multivalued operations $+^*$ and $-^*$ to
make $\BR\cup\BQ^*$ a group-like structure.  Now, we extend the values of $S$-distance to
$\BR\cup\BQ^*$. Since $g_k=\id$ for all $k\in \BZ$, we write $\sdist(a,b)=r$ for
$r\in \BR\cup\BQ^*$ if $\sdist(a,b)=r'$, where $r'$ is the unique number in
$[0,1)\cup [0,1)_{\BQ}^*$ such that $r\in r'+^*n$ for some $n\in \BZ$. Then this values depend
only on the type of $(a,b)$, that is, for $a_0,a_1,b_0,b_1\in M$, if $a_0b_0\equiv a_1b_1$,
then $\sdist(a_0,b_0)=\sdist(a_1,b_1)$ (taking values in $[0,1)\cup[0,1)_{\BQ}^*$). Then the
following fact is easily verified:
\begin{fact}\label{direct-distance}
Let $a,b,c\in M$.
\be
	\item $\sdist(b,a)=1-^*\sdist(a,b)$.
	\item $\sdist(a,c)=\sdist(a,b)+^*\sdist(b,c)$ modulo $\BZ^*$, that is, $\sdist(a,b)+^*\sdist(b,c)-^*\sdist(a,c)\subset \BZ^* $.
\ee
\noindent By (1), $\sdist$ is not symmetric, that is, for some $a,b\in M$, $\sdist(a,b)\neq\sdist(b,a)$ and so it is called a directed distance.
\end{fact}

\noindent Now, we assign to each $1$-simplex $f$ a value $n_f$ in $\BR\cup\BQ^*$ as follows.
There are $a,b\in M$ such that $[a,b]=f$; we define $n_f$ as $\sdist(a,b)$.
Then $n_f$ is well-defined, that is, it does not depend on the choice of $a$ and $b$, because if
$a_i,b_i\in M$ satisfy $[a_0,b_0]=[a_1,b_1]=f$, then $a_0 b_0\equiv a_1 b_1$ and
$\sdist(a_0,b_0)=\sdist(a_1,b_1)$. We also assign to each $1$-shell $s=f_{01}+f_{12}-f_{02}$
a multivalue $n_s$ in $\BR\cup\BQ^*$ as follows: $n_s=n_{f_{01}}+^*n_{f_{12}}-^* n_{f_{02}}$.
This value is also related to the distance of endpoints. Let $(a,a')$ be an endpoint pair
of $s$. Then $\sdist(a,a')=n_s$ modulo $\BZ^*$. Using this assignments, we give a
necessary and sufficient condition for a $1$-shell to be the boundary of a $2$-chain:

\begin{theorem}\label{NSfor[s]=0inM}
A 1-shell $s=f_{12}-f_{02}+f_{01}$ is the boundary of a 2-chain in $p$ if and only if
\begin{center}
$n_s=n_{01}+^*n_{12}+^*n_{20}\subset\BZ^*$,
\end{center}
where $n_{01}=n_{f_{01}},\ n_{12}=n_{f_{12}},\ n_{20}=-^* n_{f_{02}}$. Moreover,
it is equivalent to the condition that the two endpoints of $s$ are Lascar equivalent
over $\emptyset$.
\end{theorem}
\begin{proof}
($\Rightarrow$) Let $\alpha$ be a 2-chain with boundary $s$. By Remark \ref{chain-walk},
we may assume that $\alpha=\sum_{i=0}^{2n}\limits (-1)^i a_i$ is a chain-walk from $f_{01}$
to $f_{12}$ with $\supp(\alpha)=\{0,1,2\}$. Let $[3]=\{0,1,2\}$. By Theorem
\ref{improved_classification} and the extension axiom, there are independent elements $b$ and $d_0,d_1,\cdots,d_{2n+2}$ such that

\begin{itemize}
	\item $a_i([3])\equiv [bd_{i}d_{i+1}]_{[3]}$ if $i$ is even, and $a_i([3])\equiv [bd_{i+1}d_{i}]_{[3]}$ if $i$ is odd;
	\item For some even number $0 \le i_0\le 2n$, $[d_{i_0}d_{i_0+1}]_{\{1,2\}}\equiv f_{12}(\{1,2\})$; and
	\item For each even number $0 \le j_0\neq i_0 \le 2n$, there is an odd number $0\le j_1\le 2n$ such that $[d_{j_0}d_{j_0+1}]_{\{1,2\}}\equiv[d_{j_1+1}d_{j_1}]_{\{1,2\}}$.
\end{itemize}

\noindent Then $\sdist(d_1,d_0)+^* \sdist(d_0,d_{2n+2})=-^* n_{01}-^* n_{20}$ and by Fact \ref{direct-distance} (1), $\sdist(d_1,d_2)+^* \sdist(d_2,d_3)+^*\cdots+^* \sdist(d_{2n+1},d_{2n+2})=n+^* n_{12}$. By Fact \ref{direct-distance} (2), $\sdist(d_1,d_{dn+2})\in (-^* n_{01}-^* n_{20})\cap (n+^* n_{12})$. Since $\{0\}^*=(-^* n_{01}-^* n_{20})+^*(n_{01}+^*n_{20})$ and $(n+^* n_{12})+^*(n_{01}+^*n_{20})=n+^*(n_{01}+^*n_{12}+^*n_{20})$,
these two equations imply that $n+^*(n_{01}+^*n_{12}+^*n_{20})\subset \{0\}^*$.
Therefore, $n_{01}+^*n_{12}+^*n_{20}\subset \{0\}^*-^* \{n\}^*=\{-n\}^*$ for $n\in \BN\subset \BZ$.\\

($\Leftarrow$) Suppose $n_{01}+^*n_{12}+^*n_{20}\subset\{n\}^*$ for some $n\in \BZ$. There are independent elements
$a,b,c,a'$ such that
$$[ab]_{\{0,1\}}=f_{01}(\{0,1\}),\ [bc]_{\{1,2\}}=f_{12}(\{1,2\}),\ [a'c]_{\{0,2\}}=f_{02}(\{0,2\}).$$
So, $\sdist(a,b)=n_{01},\ \sdist(b,c)=n_{12},\ \sdist(c,a)=n_{20},$ and $\sdist(a,a')\in n_{01}+^*n_{12}+^*n_{20}$. Thus $\sdist(a,a')\in \{n\}^*$ and $\sdist(a,a')\in \{0\}^*$. Since $\{a,a'\}$ is independent, $\sdist(a,a')\in \{0\}^*\setminus \{0\}$, that is, $\sdist(a,a')=\epsilon$ or $\sdist(a,a')=1-\epsilon$.

We will find $d\in M$ such that $a\equiv_d a'$ and $d\indo abca'$, where $\indo$ is the thorn-forking independence. Consider a partial type $\Sigma(x)=\{s<\sdist(x,a)<t\leftrightarrow s<\sdist(x,a')<t\}_{s<t\in [0,1]_{\BQ}}$. Consider finitely many pairs $(s_i,t_i)$ with $s_i<t_i$ and a formula $$\bigwedge (s_i<\sdist(x,a)<t_i\leftrightarrow s_i<\sdist(x,a')<t_i).$$
We may assume $s_i \le s_0 < t_0 \le t_i$. It is enough to show that the formula
$$s_0<\sdist(x,a)<t_0\leftrightarrow s_0<\sdist(x,a')<t_0$$
is satisfiable. Suppose the formula $s_0<\sdist(x,a)<t_0$ is satisfiable.
Then, there is a pair $(s,t)$ such that $s_0<s<t<t_0$ and $s<\sdist(x,a)<t$ is satisfiable.
Let $e\in M$ be an element independent from $a$ such that $s<\sdist(e,a)<t$ holds.
Since $\sdist(a,a')\in \{0\}^*\setminus \{0\}$, there is a pair $(s',t')$ such
that $s'<t'$, $s_0<s+s'<t+t'<t_0$, and $s'<\sdist(a,a')<t'$. Then, $s<\sdist(e,a)<t$
and $s'<\sdist(a,a')<t'$
imply $s+s'<\sdist(e,a')<t+t'$. Since
$s_0<s+s'<t+t'<t_0$, $s_0<\sdist(e,a')<t_0$ and $s_0<\sdist(x,a')<t_0$ is satisfiable. By the same argument,
$s_0<\sdist(x,a')<t_0\rightarrow s_0<\sdist(x,a)<t_0$.

Therefore, there is $d\in M$ such that $\Sigma(d)$ and $\sdist(d,a)=\sdist(d,a')$. Moreover,
we may assume that $\{a,b,c,a',d\}$ is independent by taking $d\indo_{aa'}bc$.
Consider the 2-chain $\alpha=a_0+a_1-a_2$, where
\begin{itemize}
	\item $\supp(a_0)=\{0,1,3\}$, $\supp(a_1)=\{1,2,3\}$, and $\supp(a_2)=\{0,2,3\}$;
	\item $a_0(\{0,1,3\})=[a b d]_{\{0,1,3\}}$, $a_1(\{1,2,3\})=[bcd]_{\{1,2,3\}}$, and $a_2(\{0,2,3\})=[a' cd]_{\{0,2,3\}}$;
	\item $a_0\restriction\CP(\{0,1\})=f_{01}$, $a_1\restriction\CP(\{1,2\})=f_{12}$, and $a_2\restriction\CP(\{0,2\})=f_{02}$; and
	\item $a_0\restriction\CP(\{0,3\})=a_2\restriction\CP(\{0,3\})$, $a_0\restriction\CP(\{1,3\})=a_1\restriction\CP(\{1,3\})$, and $a_1\restriction\CP(\{2,3\})=a_2\restriction\CP(\{2,3\})$.
\end{itemize}

\noindent Then $\Bd\alpha=f_{01}+f_{12}-f_{02}+(a_2\restriction\CP(\{0,3\})-a_0\restriction\CP(\{0,3\}))=f_{01}+f_{12}-f_{02}$.\\

We show the `moreover' part. Let $a,a'$ be endpoints of $s$. If $a\equiv^L a'$, then $s$ is
a boundary of a $2$-chain, and $n_s\subset \{n\}^*$ for some $n\in\BZ$. Conversely, we assume
that $n_s\subset \{n\}^*$ for some $n\in \BZ$. In the proof of the right-to-left
implication, we found
$d\in M$ such that $a\equiv_d a'$. Consider the substructure
$\cl(d)=\dcl(d)=\acl(d)$ generated by $d$. Then $a\equiv_{cl(d)} a'$, and $a\equiv^L a'$.
\end{proof}
Now we are ready to prove Theorem \ref{h1_in_M_1}. Define a map
$\Phi:\ H_1(p_0)\rightarrow (\BR\cup\BQ^*)/\BZ^*$ by sending $[s]$ to
$n_s+^*\BZ^*$ (note that $(\BR\cup\BQ^*)/\BZ^*\cong \BR/\BZ$, as shown in Appendix B).
It is easy to see that
this map is surjective. Since for an endpoint pair $(a,b)$ of
$s$, $n_s+^*\BZ^*=\sdist(a,b)+^*\BZ^*$, the map $\Phi$ depends only on the endpoint pairs
of $1$-shells. By Theorem \ref{endpt_gpstr}, given $1$-shells $s_0$ and $s_1$, and endpoint
pairs $(a,b)$ and $(b,c)$ of $s_0$ and $s_1$, there is a $1$-shell $s$ such
that $[s]=[s_0]+[s_1]$ and $(a,c)$ is an endpoint pair of $s$, so the  map $\Phi$ is a
group homomorphism. Moreover, by Theorem \ref{NSfor[s]=0inM} it is injective, and therefore
it is an isomorphism. This completes the proof of Theorem \ref{h1_in_M_1}.
\section{Appendix}
\subsection{Appendix A} We show that the possible
number of equivalence classes of a bounded type-definable 
equivalence relation on a strong type is $1$ or at least $2^{\aleph_0}$. Let $T(=T^{eq})$ be any theory in
a language $\CL$ and let $\CM$ be a monster model of $T$. Fix a small subset $A=\acl(A)$, and
choose a strong type $p(x)$ over $A$ with $x$ of possibly infinite length.
\begin{theorem}
Let $E(x,y)$ be a bounded $A$-type-definable equivalence relation on $p(x)$, and denote
 the set of $E$-classes on $p$ by $p/E$. Then, \begin{center}
$|p/E|=1$ or $|p/E|\ge 2^{\aleph_0}$.
\end{center}
\end{theorem}
\begin{proof}
For convenience, we assume that $A=\emptyset$. We consider two cases:\\

Case 1. $p/E$ is finite: Let $a_0,\cdots, a_n\models p$ be
representatives of all distinct classes in $p/E$, and put $\a=(a_0,a_1,\ldots ,a_n)$.
At first, we show that $E$ is relatively definable on $p$. Consider two types
$E(x,a_0)$ and $\bigvee_{i>0}E(x,a_i)$ partitioning $p$. By compactness,
$p(x)\models E(x,a_0)\leftrightarrow \phi(x,a_0)$ for some formula $\phi(x,z)$ such that
$E(x,a_0)\models \phi(x;a_0)$. Since
$a_0\equiv a_i$, $p(x)\models E(x,a_i)\leftrightarrow \phi(x,a_i)$ for all $i\le n$. Thus,
$p(x)\wedge p(y)\models E(x,y)\leftrightarrow \psi(x,y;\a)$, where
$\psi(x,y;\z)=\bigvee_i [\phi(x,z_i)\wedge\phi(x\y,z_i)]$. Since $E$ is
invariant, $p(x)\wedge p(y)\wedge \psi(x,y;\z)\wedge \tp(\a)(\z)\models \psi(x,y;\z)(\leftrightarrow E(x,y))$.
By compactness, there is a formula $\psi'(\z)$ in $\tp(\a)(\z)$ such that
$p(x)\wedge p(y)\wedge \psi(x,y;\z)\wedge \psi'(\z)\models \psi(x,y;\a)$.
Take $\theta(x,y)\equiv \exists \z (\psi'(\z)\wedge \psi(x,y;\z))$. Then
$p(x)\wedge p(y)\models \theta(x,y)\leftrightarrow \psi(x,y;\a)$. Therefore, $E$ is relatively
definable on $p$ by the formula $\theta$. Moreover, we may assume $\theta(x,y)$ is a reflexive
and symmetric relation by replacing it with $x=y\vee (\theta(x,y)\wedge \theta(y,x))$.


Next, we find a finite $\emptyset$-definable equivalence relation $E'$ such that $p(x)\wedge p(y)\models E(x,y)\leftrightarrow E'(x,y)$. Since $E$ is an equivalence relation,
$$\begin{array}{c c l}
p(x)\wedge p(y)\wedge p(z)&\models& \bigvee_i\limits \theta(x,a_i)\wedge\bigvee_i\limits \theta(y,a_i)\wedge\bigvee_i\limits \theta(z,a_i)\\
&\wedge& \bigwedge_i \limits (\theta(x,a_i)\rightarrow \bigwedge_{i\neq j}\limits \neg \theta(x,a_j))\\
&\wedge&\bigwedge_i \limits (\theta(y,a_i)\rightarrow \bigwedge_{i\neq j}\limits \neg \theta(y,a_j))\\
&\wedge&\bigwedge_i \limits (\theta(z,a_i)\rightarrow \bigwedge_{i\neq j}\limits \neg \theta(z,a_j)) \\
&\wedge&(\theta(x,y)\wedge\theta(y,z)\rightarrow \theta(x,z)).\ (*)
\end{array}$$
Again by compactness, there is $\delta(x)\in p(x)$ such that $$\delta(x)\wedge\delta(y)\wedge\delta(z)\models (*).$$ Define a definable equivalence relation $E'(x,y)\equiv [\neg \delta(x)\wedge \neg \delta(y)]\vee [\delta(x)\wedge \delta(y)\wedge \forall z(\delta(z)\rightarrow (\theta(z,x)\leftrightarrow \theta(z,y)))]$.
\begin{claim}\label{fin_equiv_relation}
The equivalence relation $E'$ is finite.
\end{claim}
\begin{proof}
First, $\neg\delta(x)$ defines an $E'$-class. We show that on $\delta$, the $E'$-classes are of the form of $\theta(x,a_i)\wedge \delta(x)$. By the choice of $\delta$, it is partitioned by $\{\theta(x,a_i)\wedge \delta(x)\}_{i\le n}$.\\

1) We show that $\models \theta(x,a_i)\wedge \delta(x)\rightarrow E'(x,a_i)$: Choose $b\models \theta(x,a_i)\wedge \delta(x)$. Take $c\models \delta(x)\wedge \theta(x,a_i)$. Since $\theta$ is transitive on $\delta$ and $\theta(b,a_i)$ holds, $\theta(c,b)$ holds. Conversely, if $d\models \delta(x)\wedge \theta(x,b)$, then by transitivity of $\theta$ on $\delta$, $\theta(d,a_i)$ holds. Therefore, $E'(b,a_i)$ holds.\\

2) For $i\neq j$, $\neg E'(a_i,a_j)$: Suppose that for some $i\neq j$, $E'(a_i,a_j)$ holds.
Then $\theta(a_i,a_j)$ holds, but it is impossible, since $a_i,a_j\models p$ and $\theta$ coincides with
$E$ on $p\times p$.\\

\noindent By 1) and 2), the $E'$-classes are of the form $\theta(x,a_i)\wedge \delta(x)$ or
$\neg\delta(x)$, so $E'$ is a finite equivalence relation.
\end{proof}

\noindent By the proof of Claim \ref{fin_equiv_relation}, $E'$ and $E$ give the same
equivalence relation on $p\times p$. Since $E'$ is finite and $p$ is a strong type,
$p/E=p/E'$ and there is only one $E$-class in $p$.\\

Case 2. $p/E$ is infinite. Let $\kappa=|p/E|$. If $E$ is definable,
then by compactness, $|p/E|\ge \kappa'$ for any small $\kappa'$ and $E$ is not bounded.
So $E$ is not definable but type-definable; write $E(x,y)\equiv \bigwedge_{i<\lambda}\limits \phi_i(x,y)$, where each
$\phi_i(x,y)$ is a formula and $\lambda$ is an infinite cardinal. Furthermore we assume that for each $i,j<\lambda$ there is $k<\lambda$ such that $\phi_k(x,y)\equiv \phi_i(x,y)\wedge \phi_j(x,y)$. We may assume $\phi_i(x,y)$
is reflexive and symmetric (by replacing it with $x=y\vee(\phi_i(x,y)\wedge \phi_i(y,x))$) for each $i<\lambda$.
Let $\{a_k\models p\}_{k<\kappa}$ be a set of representatives of all $E$-classes.

\begin{claim}\label{inf_classes_in_phi}
For each $i<\lambda$ and $k<\kappa$, $\phi_i(x,a_k)(\CM)$ contains infinitely many $E$-classes.
\end{claim}
\begin{proof}
Fix $i<\lambda$. By compactness, there are finitely many $k_0<k_1<\cdots<k_n$ such that
$p\models \bigvee_j \phi_i(x,a_{k_j})$. By the Pigeonhole Principle, some $\phi_i(x,a_{k_l})$
contains infinitely many $a_k$'s so that $\phi_i(x,a_{k_l})$
contains infinitely many $E$-classes. Since $a_n\equiv a_m$ for all $n,m<\kappa$ and $E$ is invariant, each $\phi_i(x,a_k)$ contains infinitely many $E$-classes.
\end{proof}

\begin{claim}\label{disj_phi_in_phi}
For each $i<\lambda$ and $k<\kappa$, there are $j<\lambda$ and $k_0,k_1<\kappa$ such that $$\models \forall x[(\phi_j(x,a_{k_0})\vee\phi_j(x,a_{k_1}))\rightarrow \phi_i(x,a_k)]\wedge [\neg \exists x (\phi_j(x,a_{k_0})\wedge \phi_j(x,a_{k_1}))].$$
\end{claim}
\begin{proof}
Fix $i<\lambda$ and $k<\kappa$. By Claim \ref{inf_classes_in_phi}, $\phi_i(x,a_k)$ contains infinitely many $E$-classes. Choose two different $E$-classes in $\phi_i(x,a_k)$ and let $a_{k_0}$ and $a_{k_1}$ be representatives of two classes respectively. Since $E(x,a_{k_0})(\CM)$ and $E(x,a_{k_1})(\CM)$ 
are disjoint, by compactness, for some $j_0,j_1<\lambda$, $\phi_{j_0}(x,a_{k_0})(\CM)$ and $\phi_{j_1}(x,a_{k_1})(\CM)$ are disjoint subsets of $\phi_i(x,a_k)(\CM)$. Take $j<\lambda$ such that $\phi_j(x,y)\equiv \phi_{i}(x,y)\cap \phi_{j_0}(x,y)\cap \phi_{j_1}(x,y)$ and we are done.
\end{proof}
\noindent By Claim \ref{inf_classes_in_phi}, \ref{disj_phi_in_phi} and the fact that the
cofinality of $\lambda$ is at least $\aleph_0$, we get a binary tree
$\CB :\ 2^{< \omega} \rightarrow \omega\times \kappa$ such that for each
$b\in 2^{<\omega}$, $\CB(b \upSmallFrown 0)=(j,k_0)$ and $\CB(b\upSmallFrown 1)=(j,k_1)$,
where $j<\omega$ and $k_0,k_1<\kappa$ are given by the Claim
\ref{disj_phi_in_phi} for $(i,k):=\CB(b)$. Then, for each $\tau \in 2^{\omega}$, we obtain a set of
formulas $\{\phi_{i(\tau\restriction n )}(x,a_{k(\tau\restriction n)})\}$, where
$\CB(\tau\restriction n)=(i(\tau\restriction n ),k(\tau\restriction n))$ for each
$n\in \omega$. By the choice of $\CB$, for $\tau_0\neq \tau_1 \in 2^{\omega}$,
$\bigcap_n \phi_{i(\tau_0\restriction n )}(x,a_{k(\tau_0 \restriction n)})(\CM)$ and
$\bigcap_n \phi_{i(\tau_1 \restriction n )}(x,a_{k(\tau_1\restriction n)})(\CM)$ are disjoint,
and each of them contains at least one $E$-class. Thus, $p/E$ has at least $2^{\aleph_0}$ many elements.
\end{proof}

\subsection{Appendix B}
We shall see how to recover the ordered group  $(\BR,+)$ of real numbers from a dense linear order
extending $(\BQ,<)$ using Dedekind cuts. Consider the language
$\CL_{od,\BQ}=\{<\}\cup\{r\}_{r\in\BQ}$ and an $\CL_{od,\BQ}$-structure
$\CU=(U,<,r : r\in \BQ)$ which is a saturated dense linear order extending
$(\BQ,<)$. Then $\Th(\CU)$ has quantifier elimination.

Consider $S_1(\emptyset)$, the space of $1$-types over the empty set (which we will denote just by $S_1$). By quantifier elimination, any $1$-type $p$ is equivalent to a type of one of the following forms (where $r\in \BQ$ and $r'\in \BR\sm\BQ$):
\be
	\item $\{x=r\}$;
	\item $\{l<x<r|\ l<r\}$;
	\item $\{r<x<u|\ r<u\}$; and
	\item $\{l<x<u|\ l<r'<u\}$.
\ee
For a subset $Y\subset \BQ$, we write $Y^*:=Y\cup\{y\pm\epsilon|\ y\in Y\}$, where $\epsilon$ is an infinitesimal. So we can identify $S_1$ with the set $\BR\cup \BQ^*$ in the following way :
For $r\in \BQ$ and $r'\in \BR\sm\BQ$,
\be
	\item $\{x=r\}\leftrightarrow r$;
	\item $\{l<x<r|\ l<r\}\leftrightarrow (r-\epsilon)$;
	\item $\{r<x<u|\ r<u\}\leftrightarrow (r+\epsilon)$; and
	\item $\{l<x<u|\ l<r'<u\}\leftrightarrow r'$.
\ee
Next, we define a group-like structure on $S_1$. Define a plus-like operation $+^* : S_1\times S_1 \rightarrow \CP(S_1)$ as follows :
$$p_1 +^* p_2:=\{p|\ p\models (l_1+l_2<x<u_1+u_2),\ p_i\models l_i<x<u_i\},$$
and define a minus-like operation $-^* : S_1 \rightarrow S_1$ as follows:
$$(-^* p):=\{-u<x<-l|\ p\models l<x<u\}.$$
\noindent We extend $+^*$ and $-^*$ to operations defined on $\CP(S_1)$: For $A,B\subset S_1$, \begin{center}
$A+^* B:=\bigcup_{a\in A,b\in B}\limits a+^* b$, and $(-^* A):=\bigcup_{a\in A}\limits (-^* a)$.
\end{center}
We identify each element $a\in S_1$ with its singleton $\{a\}\in \CP(S_1)$.
Then $+^*$ and $-^*$ are commutative, associative and distributive. For
any $p_1,\cdots,p_k\in S_1$ and $k\ge 1$, we have
$$ |p_1+^* \cdots +^* p_k|\le 3.$$
We write $p_1 -^* p_2 $ for $p_1+^* (-^*p_2)$. These two notions are naturally assigned
to $\BR\cup\BQ^*$ and they are defined as follows:
\be
	\item
	\be
		\item If both $r_1$ and $r_2$ are in $\BR$ and $r=r_1+r_2$, then
		\[r_1+^* r_2:= \left \{
		\begin{array}{ll}
         \{r\} & \mbox{if $\{r\}\in \BR\setminus\BQ$}\\
         \{r\} & \mbox{if $\{r\}\in \BQ$ and $r_1,r_2\in \BQ$}\\
        \{r-\epsilon,\ r, r+\epsilon\} & \mbox{if $r\in \BQ$ and $r_1,r_2\notin \BQ$}
        \end{array}
        \right.
        \]
		\item If $r_1\in \BR\setminus \BQ$ and $r_2=q\pm\epsilon \in \BQ^*$, then $r_1+^* r_2:=\{r_1+q\}$;
		\item If $r_1\in \BQ$ and $r_2=q\pm\epsilon \in \BQ^*$, then $r_1+^* r_2:=\{(r_1+q)\pm\epsilon\}$;		
		\item If $r_1=p\pm\epsilon$ and $r_2=q\pm\epsilon\in \BQ^*$, then $r_1+^* r_2:=\{(p+q)\pm\epsilon\}$;
		\item If $r_1=p\pm\epsilon$ and $r_2=q\mp\epsilon\in \BQ^*$, then $r_1+^* r_2:=\{(p+q)-\epsilon,(p+q),(p+q)+\epsilon\}$.
	\ee
	
	\item
	\be
		\item If $r_1\in \BR$, then $-^*r_1:=-r_1$;
		\item If $r_1=p\pm\epsilon \in \BQ^*$, then $-^*r_1:=-p\mp\epsilon$.
	\ee
\ee
Now, we induce a group structure from $(S_1,+^*,-^*)$. Define an equivalence relation
$\equiv_0$ on $S_1$ by
$$p_1\equiv_0 p_2\ \mbox{iff}\ p_1-^* p_2 \subset \{0-\epsilon,\ 0,\ 0+\epsilon\},$$
and  denote by $[p]_0$ the equivalence class of an element $p\in S_1$ with respect to that
relation.
Since $\{0-\epsilon,\ 0,\ 0+\epsilon\}$ is closed under $+^*$ and $-^*$, $+^*$ and $-^*$
can be extended on $S_1/\equiv_0$. Then $(S_1/\equiv_0,+^*,-^*,[\tp(0)]_0 )$ is a group.
Actually, it is isomorphic to $(\BR,+,-,0)$.
\begin{theorem}
$(S_1/\equiv_0,+^*,-^*,[\tp(0)]_0)\cong (\BR,+,-,0)$.
\end{theorem}
\noindent Define an equivalence relation $\equiv_{\BZ}$ on $S_1$ by
$$p_1\equiv_{\BZ} p_2\ \mbox{iff}\ p_1 -^* p_2 \subset \BZ^*,$$
and denote by $[p]_{\BZ}$ the corresponding equivalence class. As above, since $\BZ^*$ is closed under
$+^*$ and $-^*$, we can extend $+^*$ and $-^*$ to on $S_1/\equiv_{\BZ}$. Then
$(S_1/\equiv_{\BZ},+^*,-^*,[\tp(0)]_{\BZ} )$ is isomorphic to $(\BR/\BZ,+,-,0)$ as a group.
\begin{theorem}
$(S_1/\equiv_{\BZ},+^*,-^*,[0]_{\BZ} )\cong (\BR/\BZ,+,-,0)$.
\end{theorem}
\noindent The equivalences $\equiv_0$ and $\equiv_{\BZ}$ are defined on $\BR\cup \BQ^*,$
$$(\BR\cup \BQ^*)/\equiv_0 \cong \BR\ \mbox{and}\ (\BR\cup \BQ^*)/\equiv_{\BZ} \cong \BR/\BZ.$$

\end{document}